\documentclass[11pt]{amsart}
\usepackage{amssymb,amsmath,xypic,comment, mathtools,paralist} 

\usepackage[margin=1in,includefoot]{geometry}
\usepackage{color} 

\newtheorem{theorem}[equation]{Theorem}
\newtheorem{corollary}[equation]{Corollary}

\newtheorem{proposition}[equation]{Proposition}

\newtheorem{definition}[equation]{Definition}   
\newtheorem{definitions}[equation]{Definitions}
\newtheorem{remark}[equation]{Remark}
\newtheorem{example}[equation]{Example}
\newtheorem{question}[equation]{Question}

\numberwithin{equation}{section}
\numberwithin{equation}{section}
\newcommand{\AR}{\mathsf{A}}
\newcommand{\BR}{\mathsf{B}}
\newcommand{\CR}{\mathsf{C}}
\newcommand{\DR}{\mathsf{D}}
\newcommand{\ER}{\mathsf{E}}
\newcommand{\FR}{\mathsf{F}}
\newcommand{\GR}{\mathsf{G}}
\renewcommand{\dim}{\mathsf{dim}}
\renewcommand{\deg}{\mathsf{deg}}
\renewcommand{\det}{\mathsf{det}}
\renewcommand{\ker}{\mathsf{ker}}
\newcommand{\lang}{\langle}
\newcommand{\rang}{\rangle}
\newcommand{\cox}{h}

\title{Chip firing on Dynkin diagrams and McKay quivers}

\author{Georgia Benkart}
\address{Department of Mathematics,
University of Wisconsin-Madison,
480 Lincoln Dr.,
Madison, WI 53706}
\email{benkart@math.wisc.edu}

\author{Caroline Klivans}
\address{Applied Mathematics \& Computer Science,
Brown University,
182 George St.,
Providence RI 02906}
\email{Caroline\_Klivans@brown.edu}

\author{Victor Reiner}
\address{School of Mathematics,
University of Minnesota,
206 Church St. SE,
Minneapolis, MN 55455}
\email{reiner@math.umn.edu}



\def\ZZ{\mathbb{Z}}
\def\NN{\mathbb{N}}
\def\CC{\mathbb{C}}
\def\RR{\mathbb{R}}
\def\QQ{\mathbb{Q}}

\def\zero{\mathbf{0}}
\def\one{\mathbf{1}}

\def\AAA{\mathfrak{A}}
\def\SSS{\mathfrak{S}}
\def\GGG{\mathfrak{G}}
\def\ggg{\mathfrak{g}}

\def\GL{\mathsf{GL}}
\def\SL{\mathsf{SL}}
\def\SO{\mathsf{SO}}
\def\SU{\mathsf{SU}}
\def\PSU{\mathsf{PSU}}
\def\exx{\mathsf{e}}
\def\KC{\mathsf{K}(C)}
\def\KG{\mathsf{K}(\gamma)}
\def\ke{\mathsf{ker}}
\def\Hom{\mathsf{Hom}}
\def\Irr{\mathsf{Irr}}
\def\coker{\mathsf{coker}}
\def\im{\mathsf{im}}
\def\height{\mathsf{ht}}

\def\stab{\mathsf{stab}}
\def\supp{\mathsf{supp}}
\def\rad{\mathsf{rad}}

\def\ab{\mathsf{ab}}
\def\ad{\mathsf{ad}}
\def\sc{\mathsf{sc}}

\begin{document} 

\begin{abstract}
Two classes of avalanche-finite matrices and their critical groups (integer cokernels) 
are studied from the viewpoint of chip-firing/sandpile dynamics, 
 namely,  the Cartan matrices of finite root systems and  the McKay-Cartan matrices for 
finite subgroups $G$ of general linear groups.
In the root system case, the recurrent and superstable configurations 
are identified explicitly
and are related to minuscule dominant weights.  In the McKay-Cartan case for finite subgroups of
the special linear group, the cokernel is 
related to the abelianization of the subgroup $G$.  In the special 
case of the classical McKay correspondence, 
the critical group and  the abelianization are shown to be isomorphic.
\end{abstract}

\maketitle


\begin{section}{Introduction}
\label{intro-section}\end{section}

The chip-firing model is a discrete dynamical system classically
modeling the distribution of a discrete commodity on a graphical
network by a Laplacian matrix.  Much early work was done in the context of the {\it abelian
sandpile model}, 
studying the avalanching dynamics of granular
flow on a grid.  The long-term stabilizing configurations exhibit a
phenomenon deemed {\it self-organized criticality}~\cite{BTW, Dhar}.  
Chip-firing dynamics and long-term behavior of the model have been
related to areas such as economic models~\cite{Biggs},
energy minimization~\cite{BakerShokrieh}, and face numbers of
matroids~\cite{Merino}.   

Recent work has established 
that much of the good behavior of abelian sandpiles 
or chip-firing models  
of graphs and directed graphs (digraphs) generalizes naturally
to certain integer matrices that have
been called {\it avalanche-finite matrices}  or {\it nonsingular $M$-matrices}  
(see, e.g.,  Guzm\'an and Klivans \cite{GuzmanKlivans} or the paper \cite{PostnikovShapiro} by Postnikov
and Shapiro,  where they
are called toppling matrices).
Assume $C$ is such a matrix in $\ZZ^{\ell \times \ell}$, and $C^t$ is its transpose.
 The  {\it critical group}  of $C$ is 
$$
\KC:=\coker(C^t: \ZZ^\ell \rightarrow \ZZ^\ell)=\ZZ^\ell / \im(C^t).
$$
The stable configurations 
alluded to above are known as the {\it critical} or {\it recurrent} configurations,  and they form a  system of distinguished 
coset representatives  for $\im(C^t)$ in $\ZZ^\ell$. 
Closely related is an alternative system of coset representatives called
the {\it superstable} configurations.
Basic facts on avalanche-finite matrices,  
recurrent and superstable configurations are reviewed in Section~\ref{chip-firing-section}.

With this in mind,  the present paper considers two kinds of matrices previously unidentified as avalanche-finite matrices.
The first is the 
Cartan matrix $C$ for a finite, crystallographic, irreducible root system,
studied in Section~\ref{root-system-section}.
Here the critical group $\KC$ has an auxiliary interpretation
as the {\it fundamental group} of the root system, that is,
the weight lattice modulo the root lattice.
Our  first main result, proven in Section~\ref{Dynkin-chips-section}, 
identifies  the superstable and recurrent configurations of $\KC$
in terms of the {\it Weyl vector} $\varrho$ 
(the half-sum of all positive roots) 
and the {\it minuscule dominant weights} $\lambda$.  \medskip

\begin{theorem}
\label{Cartan-recurrents-theorem}
For  the Cartan matrix $C$ of a  finite,  crystallographic,  irreducible
root system,  
\begin{itemize}
\item[{\rm (i)}] the superstable configurations are the zero vector 
$\mathbf{0}$ and  the minuscule dominant weights $\lambda$;
\item[{\rm (ii)}] the recurrent configurations are $\varrho$ and $\varrho-\lambda$
for all minuscule dominant weights $\lambda$.
\end{itemize}
\end{theorem}

The second kind of avalanche-finite matrix is
what we refer to as the
\emph{McKay-Cartan matrix} $C$ associated to an $n$-dimensional  faithful 
representation $\gamma: G \  \hookrightarrow \ \GL_n(\CC)$ of a finite group $G$.
Assume 
$\{\one_G=\chi_0, \chi_1,\ldots, \chi_\ell\}$ is the set of irreducible complex characters of $G$, and   
$
\chi_\gamma \cdot  \chi_i = \sum_{j=0}^\ell m_{ij} \chi_j.
$ 
Then $C$ is the $\ell \times \ell$ matrix  with $(i,j)$-entry  given by 
$c_{ij}:=n \delta_{ij} - m_{ij}$ for $1 \leq i,j \leq \ell$, where $\delta_{ij}$ is the Kronecker delta.
Our second main result is established in Section~\ref{McKay-quiver-section}. \medskip

\begin{theorem}
\label{McKay-Cartan-matrices-are-toppling-theorem}
The McKay-Cartan matrix $C$ of a faithful 
representation $\gamma: G \  \hookrightarrow \ \GL_n(\CC)$ of a finite group G
is  an avalanche-finite matrix.
\end{theorem}

As we discuss in Section \ref{rng-section},  the
abelian group $\KC$ coming from the McKay-Cartan
matrix $C$ has a {\it multiplicative} structure 
as a {\it rng} (=ring without unit).
This results from viewing $\KC$ as an {\it ideal} inside
the quotient ring  which is the 
{\it (virtual) representation ring} $R(G)$ of  $G$ with 
the principal ideal generated by $n \cdot 1 - \chi_\gamma$
factored out. Here $n$ is the degree $\gamma$,
and $\chi_\gamma$ is its character.

Section~\ref{abelianization-section} proves
the following result for faithful representations 
$\gamma:G  \ \hookrightarrow \  \SL_n(\CC)$ into special linear groups, relating $\KC$
to the \emph{abelianization} $G^\ab = G/[G,G]$   and
its {\it Pontrjagin dual} or {\it character group} $\widehat{G^{\ab}}$.

\begin{theorem} \label{abelianization-theorem}  
{For a faithful 
representation $\gamma:G \   \hookrightarrow \ \SL_n(\CC)$ of a finite group $G$,    
there is a surjection  $\KC \twoheadrightarrow \widehat{G^{\ab}}$.}
\end{theorem}
 
Section~\ref{abelianization-section} discusses examples,
including the motivating case where all our results apply: 
McKay's original correspondence \cite{McKay} for  finite subgroups $G$ of $\SL_2(\CC)$. 
McKay observed  that  the extended matrix $\tilde C = (c_{ij})$, with $c_{ij} = n \delta_{ij} - m_{ij}$ for $0 \leq i,j \leq \ell$,
for the natural 2-dimensional representation $\gamma: G \   \hookrightarrow  \SL_2(\CC)$ coming from 
the action of $G$ on $\CC^2$ by matrix multiplication coincides with an affine 
Cartan matrix for a simply-laced finite root system.  Our last result is the following.

\begin{theorem}  \label{McKay-abelianization-theorem}
For a faithful representation $\gamma:G \   \hookrightarrow \ \SL_2(\CC)$ of a finite group $G$,
 there is an isomorphism $\KC \cong \widehat{G^{\ab}}$.
\end{theorem}

Thus,  the fundamental group of a simply-laced finite root system
is (noncanonically) isomorphic to $G^\ab$ for its McKay subgroup $G$.
This turns out to be equivalent to a result of Steinberg; 
see Remark~\ref{Steinberg-abelianization-remark} below.

\section*{Acknowledgments}  This work was inspired by group discussions 
at the workshop {\it Algebraic Combinatorixx}  at the Banff International Research Station (BIRS) in 2011 and began at the workshop 
{\it Whittaker Functions, Schubert Calculus and Crystals}  at the Institute for Computational and
Experimental Research in Mathematics (ICERM) in 2013.
The authors express their appreciation to those institutes for their support and hospitality.  
The second and third authors thank the CMO-BIRS-Oaxaca Institute
for its hospitality during their 2015 workshop  {\it Sandpile Groups}.
In particular, they thank Sam Payne for a helpful discussion there
leading to the formulation of Proposition~\ref{ring-and-rng-prop},
and thank Shaked Koplewitz for allowing them to sketch his proof of 
Proposition~\ref{abelian-isomorphism-characterization}.
The third author also thanks T. Schedler for helpful comments and gratefully 
acknowledges partial support  by NSF grant DMS-1001933.

\begin{section}{Avalanche-finite matrices and chip firing}
\label{chip-firing-section}  \end{section} 

  In this section, we review some  basic notions ($Z$-matrix,
 $M$-matrix, avalanche-finite matrix, critical group,  recurrent and superstable configurations) that can be found,  for example,  in  Gabrielov \cite{Gabrielov}, Guzm\'an and Klivans
 \cite{GuzmanKlivans}, and Postnikov and Shapiro \cite[\S
   13]{PostnikovShapiro}.  
 We show  variants of certain concepts, (e.g. burning configurations), that  originated 
 in  the context of  abelian sandpile models and chip firing on graphs  can 
 be adapted to the matrix case (see Definition~\ref{burning-configuration-definition}).  In 
 Theorem \ref{thm:burning} we establish a useful relation between burning configurations 
 and  recurrent configurations for an avalanche-finite matrix.

A matrix $C= (c_{ij})$ in $\ZZ^{\ell \times \ell}$ with $c_{ij} \leq 0$ 
for all $i \neq j$ 
is called a {\it $Z$-matrix}. 
Assume $\mathbb N = \{0,1,2, \ldots\}$, and  
say that the elements $v = [v_1, \ldots, v_\ell]^t \in \mathbb N^\ell$ are  {\it (nonnegative) 
chip configurations}, viewed as assigning $v_i$ chips to state
$i$ for each  $1,2,\ldots,\ell$.   For a fixed  $Z$-matrix $C$,
call  a nonnegative chip configuration $v$  {\it stable} if 
$(0 \leq ) \, v_i < c_{ii}$ for $i=1,2,\ldots,\ell$.
If $v$ is unstable,  then  $v_i \geq c_{ii}$ for some $i$,  and a new nonnegative chip
configuration $v'$ can be created by subtracting the $i^{th}$ row of $C$ from $v$,
that is, $v':= [v'_1, \ldots, v'_\ell]^t \in \mathbb N^\ell$, where 
$v'_j:=v_j - c_{ij}$ for $j=1,2,\ldots,\ell.$ 
The  result $v'$  is referred to as  the 
{\it ($C$-)firing} or {\it ($C$-)toppling} of $v$ at state $i$.  
\begin{definition} \rm \
A $Z$-matrix is called an  {\it avalanche-finite} matrix
if every nonnegative chip configuration can be brought to a stable one
by a sequence of such topplings.
 \end{definition}
Denote the zero and all-ones vectors in $\RR^\ell$
by
$$
\begin{aligned}
\zero&:=[0,0,\ldots,0]^t,\\ 
\one&:=[1,1,\ldots,1]^t.\\
\end{aligned}
$$
For $u, v$ in $\RR^\ell$, let $u \geq v$ (resp. $u > v$) mean 
that $u_i \geq v_i$ (resp. $u_i > v_i$) for $i=1,2,\ldots,\ell$.
Call a diagonal matrix $D$ {\it positive} if $[D_{11},D_{22},\ldots,D_{\ell\ell}]^t > \zero$.

\begin{proposition}
\label{toppling-equivalences-prop}
For a $Z$-matrix $C$ in $\ZZ^{\ell \times \ell}$, 
the following conditions are equivalent:
\begin{itemize}
\item[{\rm (i)}]
$C$ is an avalanche-finite matrix.
\item[{\rm (ii)}]
$C^t$ is an avalanche-finite matrix.
\item[{\rm (iii)}]
There exists a positive diagonal matrix $D$ 
with $DC + (DC)^t$ positive definite { (that is, all its eigenvalues are positive)}.
\item[{\rm (iv)}]
The eigenvalues of $C$ all have positive real part.
\item[{\rm (v)}]
$C^{-1}$ exists and has all nonnegative entries.
\item[{\rm (vi)}]
There exists $r$ in $\RR^\ell$ with $r > \zero$ and
$Cr > \zero$.
\end{itemize}
\end{proposition}
\begin{proof}
The equivalence of (iii),(iv),(v),(vi) can be found, for example, 
in Plemmons \cite[Thm.~1]{Plemmons}.  The equivalence of (i) and (v) 
is due to Gabrielov \cite{Gabrielov}. The
equivalence of (i) and (ii) then follows.
\end{proof}
The matrices of Proposition~\ref{toppling-equivalences-prop} are commonly known as {\it  (nonsingular)  M-matrices} and arise in a broad range of mathematical disciplines.  The paper~\cite{Plemmons} by Plemmons contains $40$ equivalent conditions for a 
$Z$-matrix to be an $M$-matrix.  
 
\begin{definition} \rm \
The  {\it critical group} $\KC$ of  an avalanche-finite matrix $C$ is 
the cokernel,
$$
\KC:=\coker(C^t: \ZZ^\ell \rightarrow \ZZ^\ell):=\ZZ^\ell/\im(C^t).
$$
\end{definition} 

\begin{remark} \rm \ 
In defining $\KC=\coker(C^t)$, 
there is little danger in replacing $C$ by its transpose $C^t$ when convenient.
This is because Proposition~\ref{toppling-equivalences-prop} shows 
$C^t$ is an avalanche-finite matrix if and only if $C$ is, and nonsingular integer matrices $C$
have a (non-canonical) abelian group isomorphism
$\coker(C^t) \cong \coker(C)$, or a canonical
isomorphism 
$$\coker(C^t) \cong \widehat{\coker(C)}.$$
Here $\widehat{A}$ for a finite abelian group $A$ is its group of
characters or {\it Pontrjagin dual} \cite[Exer. 5.2.14]{DummitFoote}, 
$$
\widehat{A}:=\Hom(A,\CC^\times) \cong \Hom(A,\QQ/\ZZ),
$$
which satisfies $\widehat{A} \cong A$, but not canonically.

\end{remark}

The terminology ``critical group" for $\KC$ comes from 
certain coset representatives,
called the critical or recurrent configurations, which
are distinguished by their toppling dynamics, as we explain next.
The details about their existence and uniqueness can be found in  \cite[Lemma 13.2]{PostnikovShapiro}. 

\begin{definition} \rm \
For an avalanche-finite matrix $C$ and a nonnegative configuration $v$, 
there is a unique stable configuration,
denoted $\stab_C(v)$ and called the {\it stabilization of $v$}, 
that is reachable by a sequence of valid $C$-topplings from $v$;  moreover,  $\stab_C(v)$  is 
independent of the topplings used to reach stability.   
\end{definition}

Note that this implies
\begin{equation}
\label{stabilization-fact}
\stab_C(v+p)=\stab_C(\stab_C(v)+p)\text{ for }\   p \in \mathbb N^\ell,
\end{equation}
since any sequence of topplings that stabilize $v \longmapsto \stab_C(v)$
gives a sequence of valid topplings $v+p \longmapsto \stab_C(v)+p$,
which can be performed first when computing $\stab_C(v+p)$.

\begin{definition} \rm \ 
For each $i=1,2,\ldots,\ell$, the 
{\it $i^{th}$ avalanche operator} $X_i$ is the map on the set of stable configurations defined by $X_i(v):=\stab_C(v+e_i)$,  where $e_i$ is the
$i$th standard unit basis vector of $\ZZ^\ell$.
\end{definition}

It turns out (see \cite{Dhar} or \cite[Lemma 13.3]{PostnikovShapiro}) that
the avalanche operators commute:  $X_i X_j = X_j X_i$.
In fact, \eqref{stabilization-fact} implies the more general result, 
\begin{equation}
\label{avalanches-rewritten}
X_1^{p_1}\cdots X_\ell^{p_\ell}(v)=\stab_C(v+p),
\end{equation}
 for any vector $p=(p_1, \ldots, p_\ell)  \in \mathbb N^\ell$.  The {\it abelian sandpile model} is
 a Markov chain whose state space is the set of stable configurations and whose
 transitions are given by randomly choosing a site $i \in \{1,
 \ldots, \ell\}$ and performing the avalanche operator $X_i$.

The next proposition gives various known equivalent definitions of the recurrent configurations of the model.  For example,  (d) is the definition used in
\cite[\S13, p. 3138]{PostnikovShapiro},
while  (c) is used in \cite[Defn. 4.10]{GuzmanKlivans}, 
and (e)  is related to Dhar's burning algorithm, see Theorem~\ref{thm:burning}.  
 Moreover, (e) is the definition that we will appeal to in the proof of Theorem~\ref{Cartan-recurrents-theorem}.

To state the proposition-definition, we  introduce
the {\it support} $\supp(p):=\{i: p_i \neq 0\}$ of
a vector $p = [p_1,\dots,p_\ell]^t$ in $\ZZ^\ell$  and the {\it digraph} $D(C)$ associated
to $C$,  which has node set $\{1,2,\ldots,\ell\}$, and directed arcs
$i \rightarrow j$ whenever $c_{ij} < 0$, that is, whenever chip firing
at node $i$ adds at least one chip to node $j$.  

\begin{proposition}
\label{recurrent-definitions-prop}
 Let $C$ in $\ZZ^{\ell \times \ell}$ be an avalanche-finite matrix.
The following are equivalent for $v$ in $\ZZ^\ell$.
Define $v$ to be {\rm ($C$-)recurrent} if one of them holds (hence all of them hold):   

\begin{itemize}
\item[{\rm (a)}] $v=\stab_C(v+p)$
for some   $p \in \mathbb N^\ell$, \  $p > {\bf 0}$. 

\item[{\rm (b)}] $v=\stab_C(v+Np)$
for some  $p \in \mathbb N^\ell$, \  $p > {\bf 0}$,  and every integer $N \geq 1$.

\item[{\rm (c)}] $v=\stab_C(u)$
for some $u \in \mathbb N^\ell$ with $u_i \geq c_{ii}$ for all $i$.

\item[{\rm (d)}] $v=\stab_C(v+p_i e_i)(=X_i^{p_i}(v))$ for $i=1,2,\ldots,\ell$
for some $p = (p_1, \dots, p_\ell) \in \mathbb N^\ell$, \  $p > {\bf 0}$. 

\item[{\rm (e)}] $v=\stab_C(v+p)$
for some  $p \in \mathbb N^\ell$, with the property that 
every node $j=1,2,\ldots,\ell$ has at least one directed
path $i=i_0 \rightarrow i_1 \rightarrow \cdots \rightarrow i_{m-1} \rightarrow i_m=j$
 in the digraph $D(C)$ from a node $i$ in $\supp(p)$.

\end{itemize}
\end{proposition}   

\begin{proof}
We will check these implications:

$$
\begin{matrix}
  &               &  &         &  d  &         &\\
  &               &  &\swarrow &     &\nwarrow& \\
e &\xleftrightharpoons[\quad]{\quad}& a&         &\longrightarrow     &         &b\\
  &               &  &\nwarrow &     & \swarrow& \\  
  &               &  &         &  c  &
\end{matrix}
$$

\noindent
{\sf (a) implies (b): \,}  
Note $v=\stab_C(v+p)$ implies, 
by iterating \eqref{stabilization-fact}, that
$$
v=\stab_C(v+p)=\stab_C((v+p)+p)=\cdots=\stab_C(v+Np).
$$

\vskip.1in
\noindent
{\sf (b) implies (c): \,}  
If $p > {\bf 0}$,  then $u:=v+Np$ has $u_i \geq c_{ii}$ for large $N$.

\vskip.1in
\noindent
{\sf (c) implies (a):\,}
If $v=\stab_C(u)$,  then $v$ is stable, 
and hence $u_i \geq c_{ii} > v_i$, so one has $p:=u-v >  {\bf 0}$ with
$v=\stab(u)=\stab_C(v+p)$.

\vskip.1in
\noindent
{\sf (b) implies (d):\,}
If $v=\stab_C(v+Np)$ with $p > { \bf 0}$ and $N=1,2,...$, 
choose $N > 0$ sufficiently large so that the self-map 
$Y:=X_1^{p_1} \cdots X_{i-1}^{p_{i-1}} X_i^{p_i-1} X_{i+1}^{p_{i+1}} \cdots X_\ell^{p_\ell}$
acting on the (finite) 
set of all stable configurations has $Y^N=Y^{N+M}=Y^{N+2M}\cdots$ for some finite
order $M > 0$.  In particular,
$Y^N(v)=Y^{N+M}(v)$, and hence
$$
\begin{aligned}
v&=\stab_C(v+(N+M)p) 
 = (X_i Y)^{N+M}(v) 
 = X_i^{N+M} Y^{N+M}(v)
 = X_i^{N+M} Y^{N}(v)\\
 &= X_i^{M} (X_iY)^{N}(v)
 = X_i^{M} \stab_C(v+Np)
 = X_i^{M} (v).
\end{aligned}
$$

\vskip.1in
\noindent
{\sf (d) implies (a):\, }
Note $v=\stab_C(v+p_ie_i)$ for all $i$ implies, 
by iterating \eqref{stabilization-fact}, that
$$
v=\stab_C(v+p_1e_1)=\stab_C((v+p_1e_1)+p_2e_2)=\cdots=\stab_C(v+p).
$$

\vskip.1in
\noindent
{\sf (a) implies (e): \,}
Trivial, since $p > { \bf 0}$ means $\supp(p) = \{1,2,\ldots,\ell\}$.

\vskip.1in
\noindent
{\sf (e) implies (a):\, }
If $v=\stab_C(v+p)$ with $p$ as in (e), let 
$M:=\max_{i=1}^\ell \{c_{ii}\}$ and choose a tower of integers 
$1=:N_\ell \ll N_{\ell-1} \ll \cdots N_2 \ll N_1 \ll N_0$
where $N_{d} M < N_{d-1}$.
Then iterating \eqref{stabilization-fact} gives
$v=\stab_C(v+N_0p)$.  We claim  
$v=\stab_C(v+N_0p)$ can be computed by first ``flooding the network with chips''
as follows.   
  
Let $S_0:=\supp(p)$, and let $S_d$ for $d=1,2,\ldots,\ell-1$
denote the nodes whose shortest directed path from $\supp(p)$ has $d$ steps.
One can first do $N_1$ topplings at each node in $S_1$, 
then $N_2$ topplings at each node in $S_2$, and so on, finishing
with $N_{\ell-1}$ topplings at each node in $S_{\ell-1}$.
At the $d^{th}$ stage, each node in $S_d$  
will have received at least $N_{d-1}$ chips from 
nodes  in $S_{d-1}$, and since $N_{d-1} > N_{d} M \geq N_{d} c_{ii}$,
it will have the $N_{d} c_{ii}$ chips that it needs to
do $N_d$ valid topplings.

After these ``flooding'' topplings, the result has the form 
$v+p'$ where $p'_i \geq N_\ell=1$ for all $i$, which one can continue
toppling until stability is achieved.  
Hence $v=\stab_C(v+N_0p)=\stab_C(v+p')$, so (a) is satisfied.
\end{proof} 

\begin{theorem}\cite{Dhar}, \cite[Thm.~13.4]{PostnikovShapiro}
\label{recurrents-represent-cosets-thm}
For any avalanche-finite matrix $C$ in $\ZZ^{\ell \times \ell}$,
the recurrent configurations in $\ZZ^\ell$ form
a system of coset representatives for 
$\coker(C^t)=\ZZ^\ell/\im(C^t)$.
\end{theorem}

Closely related is  the following notion.

\begin{definition} \rm \
\label{superstable-definition}
A configuration $u \in \mathbb N^\ell$  is said to be  {\it superstable
for a $Z$-matrix $C$}   
if $z \in \mathbb N^\ell$ and $u-C^t z \in \mathbb N^\ell$ together imply that 
$z=\zero$.  
\end{definition}

When $C$ is an avalanche-finite matrix,  the superstable configurations give another set of 
coset representatives for $\KC =  \coker(C^t)=\ZZ^\ell/\im(C^t)$,  which are distinguished 
as follows (see \cite[Thm. 4.6]{GuzmanKlivans}):  
$u$ is superstable  if and only if $u$
{\it uniquely} minimizes the energy function 
\begin{equation}\label{eq:energy}
E(u):=\Vert C^{-1}u \Vert^2 = ( C^{-1}u, C^{-1}u )
\end{equation} 
among all nonnegative vectors within its coset $u+\im(C^t)$, where $(\cdot, \cdot)$ denotes the usual
inner product on $\mathbb R^\ell$.  
There is also a simple relation between the superstable configurations and
the recurrent configurations. 

\begin{theorem} \cite[Thms. 4.14, 4.15]{GuzmanKlivans}
\label{superstable-recurrent-duality}
For an avalanche-finite matrix $C = (c_{ij})$, the vector $v^C$ defined by 
\begin{equation}\label{eq:vC}
v^C:=[c_{11}-1,\ldots,c_{\ell\ell}-1]^t
\end{equation}  
has the property that $u$ in $\NN^\ell$ is  superstable 
if and only if $v^C-u$ is recurrent.  
\end{theorem}

For a given avalanche-finite matrix, in general it is hard to predict or parameterize
its set of recurrent (or superstable) configurations.  However, we will show in
Section~\ref{Dynkin-chips-section} that the Cartan matrix of a finite,
crystallographic, irreducible {\it root system} is always an avalanche-finite matrix,
and we will identify its recurrent and superstable configurations explicitly.

One way to test whether or not a configuration is recurrent
is to use a {\it burning configuration}.  The following conditions were stated by Dhar~\cite{Dhar}
for undirected graphs and by Speer~\cite{Speer} for directed graphs, (see also \cite[Thm.~ 2.27]{Perkinson}).  We give the
necessary variant for a general avalanche-finite matrix.  

\begin{definition} \rm \
\label{burning-configuration-definition}
A vector $b$ in $\NN^\ell$ is a {\rm burning configuration}
for an avalanche-finite matrix $C$ in $\ZZ^{\ell \times \ell}$ if
  \begin{itemize}  
  \item[(i)] $b$ is the image of some element of $\ZZ^\ell$ under $C^t$, and
  \item[(ii)] every node $j = 1,2,\ldots,\ell$ has at least one directed path from $i$ to $j$ in the digraph $D(C)$ from a node $i$ in $\supp(b)$  to $j$.
  \end{itemize}
  \end{definition}
  
Note it follows from  (ii) that a burning configuration $b \neq \zero$.

\begin{theorem}\label{thm:burning}
 Assume $b$  is a burning configuration for the avalanche-finite matrix $C$.   Then 
 a configuration $v$ is recurrent  if and only if $\stab_C(v+b) = v$.
\end{theorem}

\begin{proof}  ($\impliedby$ ) If $\stab_C(v+b) = v$, then $v$ is recurrent by Proposition~\ref{recurrent-definitions-prop}(e).

($\implies$)  We suppose now that $v$ is recurrent and $b = C^t z$ for $z \in \ZZ^\ell$,   and first argue that the configuration $v+b$ is unstable. 
Since $v$ is recurrent for $C$, the vector $u:= v^{C} - v$ is superstable, where $v^C$ is as in \eqref{eq:vC},  and hence
$$v^{C} - (v+b) =  u - b = u - C^{t}z$$
cannot lie in $\mathbb N^{\ell}$, that is, $v+b$ is not componentwise less than $v^{C}$, so $v+b$ is unstable.  Now we show  $\stab_C(v+b)$ is recurrent.  By definition,  it must be stable.  To see recurrence, consider expanding $b$ in terms of avalanche operators:
$$\stab_C(v+b) = X_1^{b_1}\cdots X_\ell^{b_\ell}(v).$$ 
If $v$ is recurrent, and performing a sequence of topplings on $v$  
results in a stable configuration $x$, then $x$ is also recurrent.  
Hence $\stab_C(v+b)$ is recurrent.  Since $b$ lies in $\im(C^t)$, $v$ and $\stab_C(v+b)$ are in the same coset modulo $\im(C^t)$.  But, recurrent configurations are unique per equivalence class,  and so $v$ must equal $\stab_C(v+b)$.  
  \end{proof}

\begin{remark} \rm \ 
\label{number-of-firings-remark} 
Given a burning configuration $b$ for $C$, let $z = {
  [z_1,\dots,z_\ell]^t}:=(C^t)^{-1}b$.  Then $z \in \NN^\ell$, since
$C^{-1}$ has nonnegative entries by
Proposition~\ref{toppling-equivalences-prop}.  Any stabilization
process from $v+b$ to $v$ for any recurrent $v$ has exactly $z_i$
firings of node $i$.
\end{remark}

\begin{remark} \rm \ 
\label{digraph-remark}
There is an extensive theory of critical groups for avalanche-finite matrices $C$
that come from directed graphs (see \cite{ChapmanGarciaEtAl, BjornerLovasz, HLMPPW, Speer, Wagner}).  
In it, one starts with a directed graph $D$
on node set $\{0,1,2,\ldots,\ell\}$ with $m_{ij}$ arcs directed from node $i$ to node $j$,
and  assumes  the distinguished ``source" node $0$ has at least one directed path to every other node $j$.  One then defines a {\it Laplacian matrix} $\tilde{L}$ in which $\tilde{L}_{ij}=\delta_{ij} d_i - m_{ij}$,  where $d_i$ is the outdegree of vertex $i$, and $\delta_{ij}$ is the Kronecker delta.
From this one derives the {\it reduced Laplacian} $L$ from $\tilde{L}$ by
striking out the $0^{th}$ row and column.  This always gives a 
avalanche-finite matrix; see Postnikov and Shapiro \cite[Prop. 13.1.2]{PostnikovShapiro}.
Its cokernel is the critical group $\mathsf{K}(D):=\coker(L) \cong \coker(L^t)$.
Here the cardinality $|\mathsf{K}(D)|$ also counts {\it arborescences} 
in $D$: directed trees in which every vertex has a directed path toward   
vertex $0$.
\end{remark}   

In fact, we will often obtain our 
avalanche-finite matrix $C \in \ZZ^{\ell \times \ell}$ by
striking out a row and column in a singular matrix 
$\tilde{C}$ in $\ZZ^{(\ell+1) \times (\ell+1)}$. 
Here we collect for later use  
some equivalent descriptions of the cokernels in this context. 
For  this purpose, we fix an ordered $\ZZ$-basis $\{e_0,e_1,\ldots,e_\ell\}$ for $\ZZ^{\ell+1}$.

\begin{proposition}
\label{Cartan-extended-Cartan-prop} 
Assume $C$ in $\ZZ^{\ell \times \ell}$ is obtained
from some $\tilde{C}=(c_{ij})_{i,j=0,1,\ldots,\ell}$ 
in $\ZZ^{(\ell+1) \times (\ell+1)}$ by removing the $0^{th}$ row and $0^{th}$ column.
Let $\delta=[\delta_0,\delta_1,\ldots,\delta_\ell]^t$ be a primitive
vector in $\ZZ^{\ell+1}$ in the nullspace of $\tilde{C}$, 
that is, $\tilde{C}\delta=\mathbf{0}$ with $\gcd(\delta)=1$, so that there is an inclusion of sublattices
\begin{equation}
\label{sublattice-inclusions}
\im(C^t) 
\ \  \subseteq  \ \ 
 \delta^\perp:=\{x \in \ZZ^{\ell+1}: x \cdot \delta=0\} 
\  \  \subseteq  \ \
\ZZ^{\ell+1}.
\end{equation}
Then one has
\begin{equation}
\label{perp-presentation}
\coker(\tilde{C}^t) \cong \ZZ  \oplus 
  \left( \delta^\perp/\im(\tilde{C}^t) \right).
\end{equation}
Assuming the  stronger condition that $\delta_0=1$ gives
\begin{equation}
\label{alternate-presentation}
\coker(C) \cong  
\ZZ^{\ell+1} / \big( \ZZ e_0 + \im(\tilde{C}) \big).
\end{equation}
Under  the even stronger assumption that  $\tilde{C}\delta=\mathbf{0}=\tilde{C^t}\gamma$
for some $\gamma, \delta \in \ZZ^{\ell+1}$ with $\gamma_0=\delta_0=1$, then
\begin{align}
\label{perp-cokernel-versus-cokernel}
\coker(C^t) &\cong \delta^\perp/\im(\tilde{C}^t),\\
\label{cokernel-relation}
\coker(\tilde{C}^t) &\cong \ZZ \oplus \coker(C^t).
\end{align}
\end{proposition}  

\begin{proof}
To prove \eqref{perp-presentation},
note that primitivity of $\delta$
gives surjectivity at the end of this short exact sequence
\begin{equation}
\label{delta-short-exact-sequence}
0 \rightarrow \delta^\perp 
  \longrightarrow \ZZ^{\ell+1} \overset{(-) \cdot \delta}{\longrightarrow} \ZZ 
   \rightarrow 0.
\end{equation}
The inclusions in \eqref{sublattice-inclusions}
combined with \eqref{delta-short-exact-sequence}
give this short exact sequence

$$ 0 \rightarrow \delta^\perp/\im(\tilde{C}^t) 
  \longrightarrow \ZZ^{\ell+1}/\im(\tilde{C}^t) 
   \overset{(-) \cdot \delta}{\longrightarrow} \ZZ 
   \rightarrow 0.
$$
This sequence must split,  since $\ZZ$ is a free (hence projective) $\ZZ$-module,
showing \eqref{perp-presentation}.

To see \eqref{alternate-presentation},  
observe that
$$
\ZZ e_0 + \im(\tilde{C})
=  \im \left[ \begin{matrix} 
    1 & c_{00} & c_{01} & \cdots & c_{0\ell}\\
    0 & c_{10} & c_{11} & \cdots & c_{1\ell}\\
    \vdots & \vdots & \vdots & \ddots & \vdots\\
    0 & c_{\ell0} & c_{\ell 1} & \cdots & c_{\ell \ell}\end{matrix} \right]
=  \im \left[ \begin{matrix} 
    1 & c_{01} & \cdots & c_{0\ell}\\
    0 & c_{11} & \cdots & c_{1\ell}\\
    \vdots & \vdots & \ddots & \vdots\\
    0 & c_{\ell 1} & \cdots & c_{\ell \ell}\end{matrix} \right]
= \im \left[ \begin{matrix} 
    1 & \mathbf{0}\\
    \mathbf{0} & C\end{matrix} \right], 
$$
where the second equality used $\tilde{C}\delta=\zero$ and $\delta_0=1$.
This proves \eqref{alternate-presentation}, since
$$
\ZZ^{\ell+1}/\left(\ZZ e_0 + \im(\tilde{C}) \right)
\cong \coker\left[ \begin{matrix} 
    1 & \mathbf{0}\\
    \mathbf{0} & C\end{matrix} \right] \cong \coker(C).
$$
To verify \eqref{perp-cokernel-versus-cokernel} holds, note that the projection
$$
\begin{array}{rcl}
\ZZ^{\ell+1} &\overset{\pi}{\longrightarrow} &\ZZ^\ell \\
\left[ x_0,x_1,\ldots, x_\ell \right]^t &\longmapsto & \left[x_1,\ldots,x_\ell \right]^t
\end{array}
$$
restricts to a lattice isomorphism $\pi:\delta^\perp \rightarrow \ZZ^\ell$
since $\delta_0=1$ forces $x_0=-\sum_{i=1}^\ell \delta_i x_i$ for 
$x$ in $\delta^\perp$.  We further claim that $\pi(\im(\tilde{C}^t))=\im(C^t)$, 
since
$$
\im(\tilde{C}^t)
=  \im \left[ \begin{matrix} 
    c_{00} & c_{10} & \cdots & c_{\ell,0}\\
    c_{01} & c_{11} & \cdots & c_{\ell 1}\\
    \vdots & \vdots & \ddots & \vdots\\
    c_{0\ell} & c_{1\ell} & \cdots & c_{\ell \ell}\end{matrix} \right]
=  \im \left[ \begin{matrix} 
    c_{10} & \cdots & c_{\ell,0}\\
    c_{11} & \cdots & c_{\ell 1}\\
    \vdots & \ddots & \vdots\\
    c_{1\ell} & \cdots & c_{\ell \ell}\end{matrix} \right]
=\pi^{-1} \im(C^t),
$$
where the second (resp. third)  equality used 
$\tilde{C}^t\gamma=\zero$  (resp. $ \tilde{C} \delta=\zero$).
Hence \eqref{perp-cokernel-versus-cokernel} follows.

Lastly, \eqref{cokernel-relation} follows by combining
 \eqref{perp-presentation} and
\eqref{perp-cokernel-versus-cokernel}.
\end{proof}

\begin{example} \rm \
One can compute for the matrix 
$
\tilde{C}=\left[ \begin{matrix} \ \, 30 & -15 \\  -20 & \ \, 10 \end{matrix} \right]
$
that 
$$
\coker(\tilde{C})=\coker(\tilde{C}^t) \cong \ZZ \oplus \ZZ/5\ZZ.
$$
One has $\tilde{C} \delta=0$ for
$\delta=\left[ \begin{matrix} 1\\2 \end{matrix} \right]$
having $\delta_0=1$, which agrees with
\eqref{perp-presentation}, since
$$
\delta^\perp/\im\tilde{C}^t
=\ZZ\left[ \begin{matrix} -2\\ \ \,1 \end{matrix} \right]
/\left(
\ZZ\left[ \begin{matrix} \ \, 30\\ -15 \end{matrix} \right]
+\ZZ\left[ \begin{matrix} -20\\ \  \, 10 \end{matrix} \right]
\right)
\cong\ZZ/\left(-15\ZZ+10\ZZ\right) 
= \ZZ/5\ZZ.
$$
Removing the $0^{th}$ row and column from $\tilde{C}$ gives
the matrix $C=[10] \in \ZZ^{1 \times 1}$, having
$$
\coker(C) \cong \coker(C^t) = \ZZ/10\ZZ,
$$
which agrees with \eqref{alternate-presentation}, since
$$
\ZZ^2/\left( \ZZ e_0+\im\tilde{C} \right)
=\ZZ^2/\im \left[ \begin{matrix} 1 &  30 & -15 \\ 0& 20 & \ \  10 \end{matrix} \right]
=\ZZ^2/\im \left[ \begin{matrix} 1 & 0 \\0 &  10 \end{matrix} \right]
$$
On the other hand, both \eqref{perp-cokernel-versus-cokernel},
\eqref{cokernel-relation} fail here, because $\tilde{C}^t \gamma=\mathbf{0}$,
where $\gamma=\left[\begin{matrix} 2\\3 \end{matrix}\right]$ 
and $\gamma_0 \neq 1$. 
\end{example}
 
\begin{section}{Review of root systems}
\label{root-system-section}\end{section}

Here we briefly review standard definitions and
facts about root systems  focusing
on what is needed for Sections~\ref{Dynkin-chips-section}
and \ref{abelianization-section}.
Good references are 
Humphreys \cite[\S III.9,10]{Humphreys}, and
Bourbaki \cite[\S IV.1]{Bourbaki456}.  \medskip

\subsection{Basic definitions}
For $\alpha \neq \zero$ in a real vector space $V=\RR^\ell$, with a  
positive definite inner product $(\cdot,\cdot)$,
define  
$
\displaystyle{\alpha^\vee:=\frac{2\alpha}{(\alpha,\alpha)}}
$
so that $(\alpha^\vee)^\vee=\alpha$.   Let 
$s_\alpha: V \rightarrow V$ be the {\it reflection} in the hyperplane 
perpendicular to $\alpha$ given by
$$
s_\alpha(v):=v - (v,\alpha^\vee) \alpha.
$$
Then  $s_\alpha(\alpha)=-\alpha$,  and $s_\alpha$ fixes the hyperplane
$\alpha^\perp$ pointwise.  In particular, $s_\alpha^2=\text{id}_V$.

 \begin{definition} 
A finite subset  $\Phi$ of nonzero vectors in $V$ is a {\it root system} if  \begin{itemize}
\item [\rm{(i)}]  $\Phi$ spans $V$;
\item [\rm{(ii)}]  $s_\alpha(\Phi)=\Phi$ for $\alpha \in \Phi$;
\item [\rm{(iii)}]  $\Phi \cap \RR \alpha=\{\pm \alpha\}$ for  $\alpha \in \Phi$.
\end{itemize}  \end{definition}
Condition (iii) says that the root system is \emph{reduced}.   The finite root systems considered here 
will always be assumed to be reduced.  If,  in addition, $\Phi$ satisfies 

\begin{itemize}
\item[{\rm(iv)}] $(\beta, \alpha^\vee) \in \ZZ$ for all $\alpha,\beta \in \Phi$, 
\end{itemize}
then  $\Phi$ is said to be \emph{crystallographic}.

The reflections $s_\alpha$ for $\alpha$ in a root system $\Phi$  preserve $(\cdot,\cdot)$ and generate
a subgroup $W:=\langle s_\alpha: \alpha \in \Phi \rangle$ 
of the orthogonal group $\mathsf{O}_V((\cdot,\cdot))$ called the {\it Weyl group}.
The set  $\Phi^\vee:=\{\alpha^\vee\}$  
forms another root system, called the {\it dual root system} to $\Phi$,
with $s_{\alpha^\vee}=s_\alpha$  for all $\alpha \in \Phi$, so that  $\Phi$ and $\Phi^\vee$ share the
same Weyl group $W$.

 It is  well known that a root system $\Phi$ always contains
an $\RR$-basis $\{\alpha_1,\ldots,\alpha_\ell\}$ for $V$   of {\it simple roots}
characterized by the property that $\Phi$ has a decomposition 
$\Phi = \Phi_+ \cup (-\Phi_+)$ into {\it positive} and {\it negative} roots, 
where 
$
\Phi_+:=\{\alpha \in \Phi: \alpha=\sum_{i=1}^\ell c_i \alpha_i \text{ with }c_i  \text{ in } \NN \  \text{for all}\ i \}.
$
The Weyl group acts (simply) transitively on the
sets of simple roots.

The {\it Dynkin diagram} of $\Phi$ is the graph with $\ell$ vertices corresponding to the
simple roots $\{\alpha_1,\ldots,\alpha_\ell\}$, 
with the $i$th vertex  joined to the $j$th by $(\alpha_i,\alpha_j^\vee)(\alpha_j,\alpha_i^\vee)$
edges  for $i\neq j$.   When $\alpha_i$ and $\alpha_j$ have different lengths (equivalently, when
$(\alpha_i,\alpha_i) \neq (\alpha_j, \alpha_j)$),  then an arrow is drawn
on the edges  connecting vertices $i$ and $j$ and  pointing to the shorter of the two 
roots.

The  root system $\Phi$  is {\it irreducible}  if
there does not exist a partition $\Phi =\Phi_1 \sqcup \Phi_2$
into two nonempty subsets $\Phi_1,\Phi_2$ having $(\alpha,\beta)=0$
for every $\alpha$ in $\Phi_1$ and $\beta$ in $\Phi_2$.  It is a common occurrence in the study of root systems that in order to prove a result for all root systems, it is sufficient to prove the result in the irreducible case.   Irreducibility  is 
equivalent to connectedness of the Dynkin diagram.

If the inner product $(\cdot,\cdot)$ on $V$  can be scaled so that $\alpha^\vee=\alpha$ 
for all $\alpha$ in $\Phi$, then all roots have the same length,  and $\Phi$ is said to
be {\it simply laced}.  Otherwise,  a (reduced) irreducible root system has  exactly two root lengths. 
The crystallographic condition ensures that every element of $\Phi$ is a $\mathbb Z$-linear combination
of the simple roots and that the $\mathbb{Z}$-span of $\Phi$ determines  a well-defined lattice structure $Q(\Phi)$, the {\it root lattice} of $\Phi$. 

The {\it weight lattice} of a root system $\Phi$
is 
$$
P(\Phi):=\{v \in V: (v,\alpha^\vee) \in \ZZ 
            \text{ for all }\alpha \text{ in }\Phi\}.
$$
For a crystallographic root system $\Phi$, the weight lattice 
$P(\Phi)$ contains the root lattice $Q(\Phi)$  as a (full rank) sublattice
 and has $\{\alpha_1, \dots, \alpha_\ell\}$ as an ordered $\ZZ$-basis.

\subsection{Fundamental weights}
\label{weights-subsection}
Fix an ordered set $\{\alpha_1,\ldots,\alpha_\ell\}$ of simple roots  for the root system $\Phi$.
Then the ordered set of 
{\it simple coroots} $\{\alpha_1^\vee,\ldots,\alpha_\ell^\vee\}$ in  $\Phi^\vee$ forms
a basis for $V$,
{and the {\it fundamental weights} $(\lambda_1,\ldots,\lambda_\ell)$
give the dual basis with respect to $(\cdot,\cdot)$;  that is,
$(\lambda_i,\alpha_j^\vee)=\delta_{i,j}$}.
Thus $\{ \lambda_1,\ldots, \lambda_\ell\}$  determines
an ordered $\ZZ$-basis for $P(\Phi)$,
and we will always identify 
$\ZZ^\ell$ with $P(\Phi)$
via these mutually inverse isomorphisms:
\begin{equation}\label{eq:mutinv} \begin{split}
\begin{array}{rcl} 
\ZZ^\ell & \longrightarrow & P(\Phi) \\
v= [v_1,\ldots,v_\ell]^t & \longmapsto & \lambda:=\sum_{i=1}^\ell v_i \lambda_i \\
 & & \\
P(\Phi) & \longrightarrow &\ZZ^\ell \\
\lambda & \longmapsto & 
  v:= \big[(\lambda, \alpha^\vee_1),\ldots,(\lambda,\alpha^\vee_\ell)\big]^t.\end{array}
\end{split} \end{equation}

\begin{definition}
For a crystallographic root system $\Phi$,  its
{\rm Cartan matrix} $C = (c_{ij})$ in $\ZZ^{\ell \times \ell}$ relative
to the choice of a set  $\{ \alpha_1,\ldots,\alpha_\ell\}$ of ordered simple roots
is given by  $c_{ij}:=(\alpha_i,\alpha_j^\vee)$ for $i,j=1,2,\ldots,\ell$.
Thus, the $i^{th}$ row of $C$
or $i^{th}$ column of $C^t$ expresses the simple root
$\alpha_i$ in the fundamental weights,
that is, $$\alpha_i= \sum_{j=1}^\ell c_{ij} \lambda_j.$$
\end{definition}

The cokernel $\coker(C^t)$ of the transpose $C^t$ of the Cartan matrix  is called
the {\it fundamental group} of $\Phi$, and  can be 
reinterpreted  as
$$
\coker(C^t) \cong P(\Phi)/Q(\Phi),
$$
the weight lattice modulo root lattice.
Note 
$\widehat{\coker(C^t)} 
\cong \coker(C) 
\cong P(\Phi^\vee)/Q(\Phi^\vee).
$
Their common cardinality, 
$
f:=|\coker(C^t)|=|\coker(C)|
$,
is called the {\it index of connection} for $\Phi$.
 
The fundamental weights $\{\lambda_i\}_{i=1}^\ell$ span the
extreme rays of a cone called the (closed) {\it fundamental chamber} or {\it dominant Weyl  chamber},
$$
\begin{aligned}
F&=\{v \in V: (v,\alpha) \geq 0 \text{ for all }\alpha\text{ in }\Phi_+\}\\
&=\{v \in V: (v,\alpha_i) \geq 0 \text{ for }i=1,2,\ldots,\ell \}  \\
&=\{v \in V: (v,\alpha_i^\vee) \geq 0 \text{ for }i=1,2,\ldots,\ell \} \\ 
&=\left\{\sum_{i=1}^\ell c_i \lambda_i : c_i \in \RR_{\geq 0} \right\},
\end{aligned}
$$
and elements of $F \cap P(\Phi)$ are referred to as the {\it dominant weights}. 
 
The cone $F$ forms a fundamental domain for $W$,
meaning that each $W$-orbit intersects $F$ in  exactly  one point.
For $v$ in $V$, one can define a set that quantifies 
how ``far'' $v$ is from $F$
$$
M(v):=\{ \alpha \in \Phi_+: (v,\alpha) < 0\}
$$
so that $F:=\{v \in V: M(v) = \varnothing\}$.
Then one has the following (see for example \cite[Lemma 4.5.2]{BjornerBrenti}): 
for  every $v$ not in $F$,  there is at  least one simple root $\alpha_i$
in $M(v)$, and for any simple root $\alpha_i$ in $M(v)$ 
the inequality  $| M(s_{\alpha_i}(v)) | < | M(v) |$ holds.   

\subsection{Root ordering and highest root}
\label{root-order-section}

There is an important partial ordering on the set $\Phi_+$ of positive roots of
a crystallographic root system, called the {\it root ordering}, defined by
$\alpha \leq \beta$ if $\beta-\alpha =\sum_{i=1}^\ell c_i \alpha_i$
with $c_i \geq 0$.  Here are some of its properties:

\begin{itemize}

\item[(a)] 
Every $\beta$ in $\Phi_+$ has at least one simple
root $\alpha_i$ with $(\beta,\alpha_i)>0$,
else the expansion $\beta=\sum_{i=1}^\ell k_i \alpha_i$
with $k_i \geq 0$ would give the contradiction
$0 < (\beta,\beta) =\sum_{i=1}^\ell k_i (\beta,\alpha_i) \leq 0$.
Then this $\alpha_i$  satisfies either
$\beta=\alpha_i$ or $\beta-\alpha_i$ in $\Phi_+$ (and $\beta-\alpha_i <\beta$
in root order);
see \cite[Lemma 9.4]{Humphreys}.  

\item[(b)] There are (at most) 
two roots within the fundamental chamber $F$: 
\begin{itemize}
\item  
the {\it highest root} $\tilde \alpha  :=\tilde{\alpha}(\Phi)$, which is the unique
maximum element relative to the root order on $\Phi_+$;
\item the {\it highest short root} $\alpha^*:=  \alpha^\ast(\Phi)$ in $\Phi_+$,
characterized  by the property that its corresponding coroot
$(\alpha^*)^\vee=\tilde{\alpha}(\Phi^\vee)$ is the highest root
in the dual root system $\Phi^\vee$, with respect to $\{\alpha^{\vee}_i\}_{i=1}^\ell$.  
\end{itemize}
\end{itemize}
Moreover,  $\alpha^*=\tilde{\alpha}$ if and only if $\Phi$ is  simply laced.

\begin{example} \rm \
In the root system $\Phi$ of type $\BR_\ell$ ($\ell \geq 2$), if one chooses as the
simple roots 
$$
\{e_1-e_2,e_2-e_3,\ldots,e_{\ell-1}-e_\ell,e_\ell\} 
$$
(using the identifications in \eqref{eq:mutinv}), 
then
$
\tilde{\alpha}=e_1+e_2 \neq 2e_1=\alpha^*.
$
The root order on $\BR_4$ is displayed below:
$$
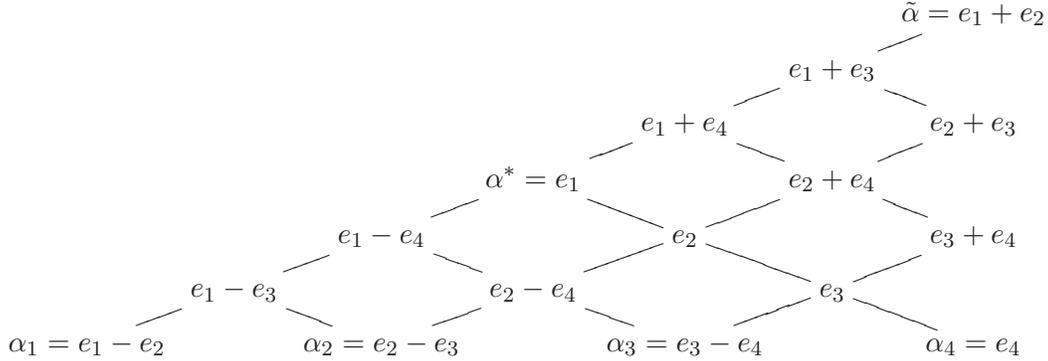
\xymatrix@C=3pt@R=6pt{
         &       &       &   &       &       &\tilde{\alpha}=e_1+e_2\\
         &       &       &   &       &e_1+e_3\ar@{-}[ur]& \\
         &       &       &   &e_1+e_4\ar@{-}[ur]& &e_2+e_3\ar@{-}[ul]\\
         &       &       &\alpha^*=e_1\ar@{-}[ur]&  &e_2+e_4\ar@{-}[ur] \ar@{-}[ul]& \\
         &       &e_1-e_4\ar@{-}[ur]& &e_2\ar@{-}[ur] \ar@{-}[ul]& &e_3+e_4\ar@{-}[ul]\\
         &e_1-e_3\ar@{-}[ur]& &e_2-e_4\ar@{-}[ur] \ar@{-}[ul]& &e_3 \ar@{-}[ur] \ar@{-}[ul]&\\
\alpha_1 =e_1-e_2 \ar@{-}[ur]& &\alpha_2 =e_2-e_3 \ar@{-}[ur] \ar@{-}[ul]& &\alpha_3 =e_3-e_4\ar@{-}[ur] \ar@{-}[ul]& &\alpha_4 =e_4\ar@{-}[ul]
}
$$
\end{example}

\noindent
Note  it is a consequence of the fact that
$\tilde{\alpha},\alpha^*$ belong to $F$ that their  expansions relative to 
fundamental  weights have nonnegative coefficients:
\begin{equation}
\label{positive-expansions-for-highest-roots}
\begin{aligned}
\tilde{\alpha} &= \sum_{i=1}^\ell q_i \lambda_i, 
  \quad \text{ with }q:=[q_1,\ldots,q_\ell]^t \geq \zero,\\
\alpha^* &= \sum_{i=1}^\ell q^*_i \lambda_i, 
  \quad \text{ with }q^*:=[q^*_1,\ldots,q^*_\ell]^t \geq \zero.\\
\end{aligned}
\end{equation}

\subsection{ Extended Cartan matrix}
\label{extended-Cartan-matrix-section}

\begin{definition}
For a finite, crystallographic, irreducible root system $\Phi$,
 the {\rm extended Cartan matrix} $\tilde{C}$ in $\ZZ^{(\ell+1) \times (\ell+1)}$
is given by $$\tilde{C}_{ij}:=(\alpha_i,\alpha_j^\vee)$$ for $i,j=0,1,2,\ldots,\ell$,
where $\alpha_0:=-\tilde{\alpha}$.
\end{definition}

While the Cartan matrix $C$ is nonsingular, the extended Cartan matrix
will, by its definition, have these left/right nullspaces:
\begin{equation} \label{eq:nullspace}
\begin{array}{rllll}
\text{If } \ \tilde{\alpha}& =\tilde{\alpha}(\Phi)=&\displaystyle \sum_{i=1}^\ell \delta_i \alpha_i,
 &\text{ then } \ker(\tilde{C}^t)=\RR \delta 
    &\text{ where }\delta:=[1,\delta_1,\ldots,\delta_\ell]^t \in \mathbb N^{\ell+1}.\\
\text{If } \alpha^*(\Phi^\vee)& =\tilde{\alpha}(\Phi)^\vee=&\displaystyle \sum_{i=1}^\ell \phi_i \alpha^\vee_i,
  &\text{ then } \ker(\tilde{C})=\RR \phi 
    &\text{ where }\,\phi:=[1,\phi_1,\ldots,\phi_\ell]^t \in \mathbb N^{\ell+1}.
\end{array}
\end{equation} 
The  coordinates of this last vector $\phi$ give the ``marks''  labeling the nodes on the affine diagrams
in Kac \cite[Ch. 4, Table Aff 1]{Kac}.
When $\Phi$ is simply laced,  then 
$\tilde{\alpha}(\Phi)=\tilde{\alpha}(\Phi^\vee)
=\alpha^*(\Phi)=\alpha^*(\Phi^\vee)$
and $\delta=\delta^\vee=\phi$.
\medskip

\subsection{Weyl vector and minuscule weights}

 There is a distinguished element $\varrho$ in the fundamental chamber $F$,  the so-called  {\it Weyl vector}, which
is  the half-sum of all the positive roots.  Since for each  simple root $\alpha_i$, the corresponding reflection $s_{\alpha_i}$ permutes the set $\Phi_+ \setminus \{\alpha_i\}$, one has 
$$\varrho - (\varrho, \alpha_i^\vee) \alpha_i = s_{\alpha_i}(\varrho) = \frac{1}{2}\sum_{\alpha \in \Phi_+\setminus \alpha_i} \alpha \, -\, \frac{1}{2}\alpha_i = \varrho - \alpha_i,$$
implying that $(\varrho, \alpha_i^\vee) = 1$ for all $i$.   Then it follows from the duality of the bases $\{\alpha_i^\vee\}$ and
$\{\lambda_i\}$ that the expansion for $\varrho$ in terms of the fundamental weights  is given by 
\begin{equation}
\label{half-sums-of-roots-coroots}
\varrho:=\varrho(\Phi)
  := \frac{1}{2} \sum_{\alpha \in \Phi_+}  \alpha
        = \sum_{i=1}^\ell \lambda_i.
        \end{equation}

\begin{definition} A weight $\lambda$ in $P(\Phi)$ is {\rm minuscule} if 
$(\lambda,\alpha^\vee)$ lies in $\{-1,0,1\}$ for all $\alpha$ in $\Phi$.
\end{definition}
 
The minuscule weights $\lambda$ which are {\it dominant}, 
that is, lie in $F$, 
have important properties.  We collect some of
them here, compiling exercises from
Bourbaki \cite[Exer. VI.1.24(a,c),  VI.2.2, VI.2.5(a,d)]{Bourbaki456}
and
Humphreys \cite[Exer. III.13.13]{Humphreys}, 
some worked out by Stembridge in \cite[\S 1.2]{Stembridge}.

\begin{proposition}
\label{minuscule-dominant-properties}
Let $\Phi$ be a finite, crystallographic, irreducible root system.

\begin{enumerate}

\item[{\rm (a)}] 
The minuscule dominant weights
are exactly the fundamental weights $\lambda_i$ having coefficient
$\delta_i^\vee=1$ in this expansion:
\begin{equation}
\label{minuscule-dominant-characterization}
\tilde{\alpha}(\Phi^\vee)=\sum_{i=1}^\ell \delta^\vee_i \alpha_i^\vee.
\end{equation}

\item[{\rm(b)}] 
There are $f-1$ minuscule dominant weights if 
 $f=|P(\Phi)/Q(\Phi)|$ is the index of connection.

\item[{\rm(c)}] 
The zero vector $\mathbf{0}$ together with
the minuscule dominant weights give a system of
coset representatives for the nonzero cosets 
in $\coker(C^t)=P(\Phi)/Q(\Phi)$.

\item[{\rm(d)}] 
The zero vector $\mathbf{0}$ and
the minuscule dominant weights can be
characterized as follows:  each is the
unique element $\lambda$ in $F \cap P(\Phi)$ 
in its coset $\lambda+Q(\Phi)$
which is minimal in the {\it root
ordering} on $F \cap P(\Phi)$:
if $\mu$ in $F \cap P(\Phi)$
has $\lambda-\mu=\sum_{i=1}^\ell z_i \alpha$ with $z_i$ in $\NN$,
then $\mu=\lambda$.
\end{enumerate}
\end{proposition}

\begin{section}{Chip firing with Cartan matrices}
\label{Dynkin-chips-section}\end{section} 

Fixing a finite, crystallographic, irreducible root system $\Phi$,
and a choice of simple roots $\alpha_1,\ldots,\alpha_\ell$,
we wish to consider its Cartan matrix $C$ in $\ZZ^{\ell\times \ell}$ 
as a $Z$-matrix, and do topplings with respect to $C$.
If $e_1,\ldots, e_\ell$ are  the standard unit basis vectors in $\ZZ^\ell$,
then under our usual root system identification  
$\ZZ^\ell \longleftrightarrow P(\Phi)$ in \eqref{eq:mutinv}  sending 
$v \mapsto \sum_{i=1}^\ell v_i \lambda_i$, 

\begin{itemize}
\item one identifies $e_i=\lambda_i$,
so that an avalanche operator acts as $X_i(v)=v + \lambda_i$, 
\item one identifies  
$q=\tilde{\alpha}$ and $q^*=\alpha^*$
where $q, q^*$ are defined in \eqref{positive-expansions-for-highest-roots},
\item one identifies $\one = \varrho$ due to 
\eqref{half-sums-of-roots-coroots}, 
\item a firing/toppling at some node $i \in \{1,\dots, \ell\}$ attempts to replace
$v$ with $v-\alpha_{i}$, using the 
interpretation of the entries of  row $i$ of $C$  as the coordinates of $\alpha_i$ with respect
to the fundamental weights, 
but it gives a {\it valid} toppling if and only if $v_{i}=(v,\alpha_{i}^\vee) \geq 2=c_{ii}$.  
\end{itemize}

\subsection{Identifying the recurrent,  superstable,  and burning configurations of a Cartan matrix}
The following proposition seems  well known, 
but is only implicit in 
Postnikov and Shapiro \cite[Proof of Prop.~13.1]{PostnikovShapiro} and
proven by Kac \cite[Chap.~4]{Kac} in different language.

\begin{proposition} The
Cartan matrix $C$ of  a finite, crystallographic, irreducible
root system is an avalanche-finite matrix.
\end{proposition}
\begin{proof}
By Proposition~\ref{toppling-equivalences-prop}(ii,v), it suffices to 
note that the coefficient vector $r (> \zero)$ of
$\varrho= \sum_{i=1}^\ell r_i \alpha_i$ has the entries of $r^t C$ expressing 
$\varrho$ in the basis $\{\lambda_i\}_{i=1}^\ell$, and hence 
$r^tC=\one (> \zero)$ by \eqref{half-sums-of-roots-coroots}.\footnote{An alternative way to prove this is to use the fact that the matrix $C^{-1}$ expressing the $\lambda_i$ in terms of the simple roots has nonnegative coefficients
as seen in \cite[Table 1, Sec. 13.2]{Humphreys}. Of course, the two statements are
equivalent by  Proposition~\ref{toppling-equivalences-prop}.}
\end{proof} 

\begin{proposition}\label{P:burn}  The burning configurations for a
Cartan matrix $C$ of  a finite, crystallographic, irreducible
root system $\Phi$ are the nonzero elements $b$ of the fundamental chamber $F$ that lie in the root lattice $Q(\Phi)$.
\end{proposition}

\begin{proof}   If $b \in F \cap Q(\Phi)$,   then $b$ is  the image under  $C^t$  of an element of $\ZZ^\ell$ (namely, its coefficient vector relative to the simple roots), 
and  $b \geq \zero$ because $b$ is an element of $F$.     When $b \neq \zero$, there is 
a path from node $i$ in $\supp(b)$ to any node $j$  of the Dynkin diagram, because the diagram is connected. 
The conditions of Theorem \ref{thm:burning} are satisfied, and $b$ is a burning configuration.
Conversely, if $b \in \ZZ^\ell$ is a burning configuration for $C$, then $b \geq \zero$, and $b =  C^t z$  for some $z \in \ZZ^\ell$ so that $b \in F \cap Q(\Phi)$.   \end{proof}

\begin{remark} \rm \
\label{high-roots-are-burning-remark}
The highest root $\tilde \alpha$ and highest short root $\alpha^*$ of $\Phi$ lie in $F \cap Q(\Phi)$
and hence give burning configurations for any  finite, crystallographic, irreducible
root system $\Phi$.      When $\det(C) = 1$,  (i.e. when $\Phi$ is of type $\ER_8$, $\FR_4$, or $\GR_2$), so that $P(\Phi)=Q(\Phi)$,
the burning configurations for $\Phi$ are exactly the dominant weights $\big(F \cap P(\Phi)\big)\setminus \{\zero\}$. 

 When using the highest root $\tilde \alpha = \sum_{i=1}^\ell \delta_i\alpha_i$  as a burning configuration, 
Remark~\ref{number-of-firings-remark} implies that
the number of topplings from $v+\tilde{\alpha}$ to $v$ will be $\height(\tilde \alpha) := \sum_{i=1}^\ell \delta_i  = \cox-1$, 
where $\cox$  is the {\it Coxeter number},   
$$
{\cox}: = \frac{|\Phi|}{\ell},
$$
 which is also equal to the order of a {\it Coxeter element}
$w = s_{\alpha_1} \cdots s_{\alpha_\ell} \in W$. (See for example, \cite[Sec.~3.18]{Humphreys2}, for some of its other incarnations.)
Similarly, when $\alpha^*$ is used as a burning
configuration, the number of topplings from $v + \alpha^*$ to $v$ is given by the height  $\height(\alpha^*)$ of $\alpha^*$.
\end{remark} 

Our main result of this section is 
Theorem~\ref{Cartan-recurrents-theorem}, which  
identifies the recurrent and superstable
configurations for a Cartan matrix $C$.

\vskip.1in
\noindent
{\bf Theorem~\ref{Cartan-recurrents-theorem}.}
{\it 
The Cartan matrix $C$ of a finite, crystallographic,  irreducible
root system has as its
\begin{itemize}
\item[{\rm (i)}] superstable configurations the zero vector
$\mathbf{0}$ and the minuscule dominant weights $\lambda_i$, 
\item[{\rm (ii)}] recurrent configurations $\one=\varrho$ and 
$\one-e_i=\varrho-\lambda_i$
for all minuscule dominant weights $\lambda_i$.
\end{itemize}
}
 
 \medskip
 
The assertions (i) and (ii)  in the theorem are equivalent via
Theorem~\ref{superstable-recurrent-duality}, since $c_{ii}=2$ for $i=1,2,\ldots,\ell$, so that $v^C=\one=\varrho$.  We will give two proofs:
the first proves (i) and has the advantage of brevity, while the second
proves (ii) and has the advantage of identifying certain
stabilization sequences with other known combinatorial
objects-- see Remarks~\ref{max-chains-in-root-order-remark} and \ref{numbers-game-remark} below.

\begin{proof}[Proof of Theorem~\ref{Cartan-recurrents-theorem} 
using (i)]
Our identification of $\ZZ^\ell$ with the weight lattice
$P(\Phi)$ identifies
$\NN^\ell$ with the set $F \cap P(\Phi)$ of dominant weights.  Recall
from Definition~\ref{superstable-definition} that the superstable configurations
are the $u$ in $\NN^\ell$ for which $z \in \NN^\ell$ together 
with $u-C^tz$ lying in $\NN^\ell$ forces $z=\zero$.  Such $u$ then correspond to $\lambda$ in  $F \cap P(\Phi)$ for which
$z =  [z_1,\dots, z_\ell]^t \in \NN^\ell$ together with
 $\lambda-\sum_{i=1}^\ell z_i \alpha_i=:\mu$ lying in $F \cap P(\Phi)$
forces $\lambda=\mu$.  But this is exactly the characterization of  $\mathbf{0}$
and the minuscule dominant weights given by
Proposition~\ref{minuscule-dominant-properties}\,(d).
\end{proof}

\begin{proof}[Proof of Theorem~\ref{Cartan-recurrents-theorem} using (ii)]   
We first explain why the assertions in   (ii)  follow once we show
\begin{equation}\label{rho-is-recurrent}  \varrho=\stab_C(\varrho+\tilde{\alpha}),  \quad \hbox{\rm and} \end{equation} 
\begin{equation}\label{rho-minus-minuscule-is-recurrent}  \varrho-\lambda_i =\stab_C\left( (\varrho-\lambda_i) +\alpha^*\right)\end{equation} for $\lambda_i$ any minuscule dominant weight.   
 
Indeed, since Remark~\ref{high-roots-are-burning-remark} shows
that $\tilde{\alpha}$ and $ \alpha^*$ are both burning configurations,
\eqref{rho-is-recurrent} and  \eqref{rho-minus-minuscule-is-recurrent}
would  show that the elements in (ii) are $C$-recurrent. 
As the set of elements in (ii)
has cardinality $f=|\coker(C^t)|$ from 
Proposition~\ref{minuscule-dominant-properties}(b),
we would then know from 
Theorem~\ref{recurrents-represent-cosets-thm} that there
are no other $C$-recurrent configurations.
\vskip.1in
\noindent
{\sf Proof of  \eqref{rho-is-recurrent}.}
Use Property (a) of the root ordering from Section~\ref{root-order-section} 
repeatedly to express the highest root  as
$$
 \begin{array}{rllll}
\varrho+\tilde{\alpha}
 &=&\varrho +\beta_{\cox-1}
 &=&\varrho+ \alpha_{i_1}+\alpha_{i_2}+ \cdots 
 +\alpha_{i_{\cox-2}}+\alpha_{i_{\cox-1}}\\
&\overset{{\text{topple }i_{\cox-1}}}{\longrightarrow}&
  \varrho +\beta_{\cox-1}
 &=&\varrho+ \alpha_{i_1}+\alpha_{i_2}+ \cdots 
 +\alpha_{i_{\cox-2}} \\
& \vdots\\
&\overset{{\text{topple }i_3}}{\longrightarrow}&
  \varrho +\beta_{2}
 &=&\varrho+ \alpha_{i_1}+\alpha_{i_2}\\
&\overset{{\text{topple }i_2}}{\longrightarrow}&
  \varrho +\beta_{1}
 &=&\varrho+ \alpha_{i_1}\\
&\overset{{\text{topple }i_1}}{\longrightarrow}&
   \varrho.
  & &
\end{array}
$$
The reason this works is 
$
v=\varrho +\beta_{k} 
$
has
$
v_{i_k}=(\varrho +\beta_{k}, \alpha_{i_k}^\vee)
=(\varrho, \alpha_{i_k}^\vee) +(\beta_{k}, \alpha_{i_k}^\vee) \geq 2,
$
since $(\varrho, \alpha_{i_k}^\vee)=1$ by \eqref{half-sums-of-roots-coroots}, 
and the integer $(\beta_k, \alpha^\vee_{i_k})$ is strictly positive 
by construction.  Thus,  $\varrho = \stab_C(\varrho + \tilde \alpha)$ as claimed.

\vskip.1in
\noindent
{\sf Proof of  \eqref{rho-minus-minuscule-is-recurrent}.}
Fix a minuscule dominant weight $\lambda$, that is, a minuscule
weight lying in the fundamental chamber $F$.  Section~\ref{root-order-section}
characterizes $\lambda$ as a fundamental weight $\lambda_i$ 
for which \eqref{minuscule-dominant-characterization} gives 
$
1=\delta^\vee_i
=(\lambda_i,\tilde{\alpha}(\Phi^\vee))
=(\lambda_i,(\alpha^*)^\vee).
$
Consequently, one has
$$
s_{\alpha^*}(\lambda)=\lambda-(\lambda,(\alpha^*)^\vee) \alpha^*
=\lambda-\alpha^*.
$$
Now set $u^{(0)}:=s_{\alpha^*}(\lambda)$.  Using induction
on the cardinality of the set $M(u^{(k)})$ defined in 
Section~\ref{weights-subsection}, 
we can create a sequence
$u^{(0)},u^{(1)},\ldots,u^{(m)}$ eventually ending with $u^{(m)}$ in $F$, 
where $m = \height(\alpha^*)$ and $u^{(k)}= s_{\alpha_{i_k}}(u^{(k-1)})$ for some simple root
$\alpha_{i_k}$ having $(u^{(k-1)},\alpha_{i_k})<0$.
Note that at each stage $k=1,2,\ldots,m$, 
one has  $u^{(k)}=w_{(k)}(\lambda)$ for an element 
$
w_{(k)}:= s_{\alpha_{i_k}} \cdots  s_{\alpha_{i_2}} s_{\alpha_{i_1}} s_{\alpha^*}
$
of the Weyl group $W$.  Therefore, the inner product
$$
(\alpha_{i_k}^\vee, u^{(k-1)})
=(\alpha_{i_k}^\vee, w_{(k)}(\lambda) )
=(w_{(k)}^{-1}(\alpha_{i_k}^\vee) , \lambda)
$$
always lies in $\{-1,0,1\}$, because $\lambda$ is minuscule.  
However, $(\alpha_{i_k},u^{(k-1)})<0$ by construction, so
\begin{equation}
\label{firing-minus-one-in-numbers-game}
(\alpha_{i_k}^\vee, u^{(k-1)})=-1.
\end{equation}
Therefore 
$
u^{(k)}= s_{\alpha_{i_k}}(u^{(k-1)})
=u^{(k-1)} -( u^{(k-1)}, \alpha_{i_k}^\vee)\alpha_{i_k}
=u^{(k-1)}+\alpha_{i_k}.
$
Thus, we obtain a sequence
\begin{equation}
\label{numbers-game-winning-sequence}
\begin{array}{rllll}
u^{(0)}&:=&s_{\alpha^*}(\lambda)&=&\lambda-\alpha^* \\
u^{(1)}&:=&s_{i_1} s_{\alpha^*}(\lambda)
         &=&\lambda-\alpha^*+\alpha_{i_1} \\
u^{(2)}&:=&s_{i_2} s_{i_1} s_{\alpha^*}(\lambda)
         &=&\lambda-\alpha^*+(\alpha_{i_1}+\alpha_{i_2}) \\
\vdots& & & & \\
u^{(m)}&:=&s_{i_m} \cdots s_{i_2} s_{i_1} s_{\alpha^*}(\lambda)
         &=&\lambda-\alpha^*+(\alpha_{i_1}+\alpha_{i_2}+\cdots+\alpha_{i_m})=\lambda,
\end{array}
\end{equation}
where $u^{(m)}=\lambda$ follows,  because $u^{(m)}$ lies in 
both the $W$-orbit of $\lambda$ and the  fundamental chamber  $F$.

We then claim $v^{(k)}:=\varrho-u^{(k)}$ for $k=0,1,2,\ldots,m$
gives this valid $C$-toppling sequence, showing
\eqref{rho-minus-minuscule-is-recurrent}:
$$
\begin{array}{rlll}
v^{(0)}&=&\varrho-\lambda+\alpha^*,\\
\overset{\text{topple }i_1}{\longrightarrow}
 v^{(1)}&=&\varrho-\lambda+\alpha^*-\alpha_{i_1},\\
\overset{\text{topple }i_2}{\longrightarrow}
 v^{(2)}&=&\varrho-\lambda+\alpha^*-(\alpha_{i_1}+\alpha_{i_2}),\\
\vdots& \\
\overset{\text{topple }i_m}{\longrightarrow}
 v^{(m)}&=&\varrho-\lambda+ \alpha^*
           -(\alpha_{i_1}+\alpha_{i_2}+\cdots+\alpha_{i_m})
=\varrho-\lambda.\\
\end{array}
$$
These are valid topplings,  because for each $k=1,2,\ldots,m$,
one has using \eqref{firing-minus-one-in-numbers-game} 
$$
v_{i_k}^{(k-1)}
= (\alpha_{i_k}^\vee, v^{(k-1)}) 
= (\alpha_{i_k}^\vee, \varrho - u^{(k-1)}) 
=(\alpha_{i_k}^{\vee},\varrho)-(\alpha_{i_k}^{\vee},u^{(k-1)})
=1-(-1)=2. \qedhere
$$

\end{proof}

\begin{example}   \rm \
\label{first-E6-example} 
When $\Phi$ is of  type $\ER_6$, there are two minuscule dominant weights $\lambda_i$, 
whose associated nodes in the Dynkin diagram are 
darkened here:
 $$
\xymatrix@C=4pt@R=6pt{
& & \circ_{\tiny 2} \ar@{-}[d]& & \\
 \bullet_{\tiny 1} \ar@{-}[r]  &  \circ_{\tiny 3} \ar@{-}[r] &  \circ_{\tiny 4} \ar@{-}[r] &  \circ_{\tiny 5} \ar@{-}[r] &  \bullet_{\tiny 6}    
}. 
$$
This labeling of the Dynkin diagram corresponds to the following Cartan matrix:
\[C = \begin{pmatrix*}[r]
  2 & 0 & -1 & 0 & 0 & 0 \\
  0 & 2 & 0 & -1 & 0 & 0 \\
  -1 & 0 & 2 & -1 & 0 & 0 \\
  0 & -1 & -1 & 2 & -1 & 0 \\
  0 & 0 & 0 & -1 & 2 & -1\\
  0 & 0 & 0 & 0 & -1 & 2 
\end{pmatrix*}
\]   
Let  $\lambda$  be the left minuscule dominant weight, represented by the vector $[1,0,0,0,0,0]^t$.
We
exhibit the $C$-toppling sequence $v^{(0)}, v^{(1)}, \ldots, v^{(m)}$ 
showing $\varrho-\lambda=\stab_C((\varrho-\lambda)+\alpha^*)$, where 
$$
\alpha^* = \tilde \alpha  =
\alpha_1 +2\alpha_2 + 2 \alpha_3 + 3\alpha_4 + 2 \alpha_5 + \alpha_6 = \lambda_2.$$   
We display the coordinates relative to the basis of fundamental weights $\{\lambda_i\}_{i=1}^6$ using the 
Dynkin diagram. The number of steps (including addition of $\alpha^*= \tilde \alpha$  is $12$,  which is the Coxeter
number for $\ER_6$.  
$$
\begin{array}{rccccc}

\varrho-\lambda=&
\left[
\begin{matrix}
  & &1& & \\
0 &1&1&1& 1    
\end{matrix}
\right]

&\overset{+\alpha^*}{\longrightarrow}&

\left[
\begin{matrix}
 & &2& & \\
0&1&1&1&1    
\end{matrix}
\right]

&\longrightarrow&

\left[
\begin{matrix}
& &0 & & \\
 0   &  1  &  2  &  1  &  1    
\end{matrix}
\right]

\\
& & &  & & \\
&\quad\uparrow\quad & &  & &\quad\downarrow\quad \\
& & &  & & \\


&
\left[
\begin{matrix}
& &1 & & \\
 2   &  0  &  1  &  1  &  1    
\end{matrix}
\right]

& &  & & 

\left[
\begin{matrix}
& &1 & & \\
 0   &  2  &  0  &  2  &  1    
\end{matrix}
\right]

\\
& & &  & & \\
&\quad\uparrow\quad & &  & &\quad\downarrow\quad \\
& & &  & & \\


&\left[
\begin{matrix}
& &1 & & \\
 1   &  2  &  0  &  1  &  1    
\end{matrix}
\right]

 & &  & & 

\left[
\begin{matrix}
& &1 & & \\
 1   &  0  &  1  &  2  &  1    
\end{matrix}
\right]

\\

& & &  & & \\
&\quad\uparrow\quad & &  & &\quad\downarrow\quad \\
& & &  & & \\


&\left[
\begin{matrix}
& &0 & & \\
 1   &  1  &  2  &  0  &  1    
\end{matrix}
\right]

 & &  & & 

\left[
\begin{matrix}
& &1 & & \\
 1   &  0  &  2  &  0  &  2    
\end{matrix}
\right]

\\
& & &  & & \\
& \quad\uparrow\quad & &  & &\quad\downarrow\quad \\
& & &  & & \\


&\left[
\begin{matrix}
& &2 & & \\
 1   &  1  &  1  &  0  &  1    
\end{matrix}
\right]

&\leftarrow&

\left[
\begin{matrix}
& &2 & & \\
 1   &  1  &  0  &  2  &  0    
\end{matrix}
\right]

&\leftarrow& 

\left[
\begin{matrix}
& &1 & & \\
 1   &  0  &  2  &  1  &  0    
\end{matrix}
\right]
\end{array}
$$

\end{example} 

\medskip

\begin{remark} \rm \ 
\label{who-is-the-identity-remark}
Since the elements of $\coker(C^t)=P(\Phi)/Q(\Phi)$ are
represented by the set of all recurrent configurations 
$\{\varrho, \  \varrho-\lambda_i: i \text{ a minuscule node}\}$,
one might guess that the zero coset \break  $\zero+\im(C^t)$
is represented by $\varrho$.  This is {\it not} always true, e.g. in type 
$\AR_\ell$ for $\ell$ odd, the zero coset is represented 
by $\varrho-\lambda_{\frac{1}{2}(\ell+1)}$.  The question of which
recurrent configuration represents the zero coset is
equivalent to determining which element among $\zero$ and
the minuscule dominant weights $ \lambda_i$ is equivalent 
to $\varrho$ in $P(\Phi)$ modulo $Q(\Phi)$.
\end{remark}

\begin{remark} \rm \
\label{max-chains-in-root-order-remark}
The proof of  \eqref{rho-is-recurrent} shows  that stabilization sequences
from $\varrho+\tilde{\alpha}$ to $\varrho$ have an obvious bijection with
maximal chains
$$
\beta_1 < \beta_2 < \cdots < \beta_{\cox-1}=\tilde{\alpha}
$$
in the root ordering on $\Phi_+$.
\end{remark}

\begin{remark} \rm \
\label{numbers-game-remark}   The proof of \eqref{rho-minus-minuscule-is-recurrent}
exhibits  stabilization sequences from $(\varrho-\lambda)+\alpha^*$ to $\varrho-\lambda$, in which 
a sequence of
vectors in \eqref{numbers-game-winning-sequence}, 
$$
s_{\alpha^*}(\lambda)=u^{(0)},u^{(1)},\ldots,u^{(m-1)},u^{(m)}=\lambda,
$$
$(m = \height(\alpha^*))$  is subtracted.   We explain here how this sequence
\eqref{numbers-game-winning-sequence}
 may be viewed as a {\it winning sequence} in Mozes's {\it numbers game} 
using the Cartan matrix $C$ for $\Phi$, starting from the play position 
$u^{(0)}$.
  In this game, 
if $u=[u_1,\ldots,u_\ell]^t=\sum_{i=1}^\ell u_i \lambda_i$ has any coordinate
$u_i =(u,\alpha_i^\vee)< 0$, then one is permitted to do a {\it numbers firing} at node $i$ that replaces 
$
u \longmapsto u'=s_{\alpha_i}(u),
$
(or in coordinates, $u'_j=u_j-c_{ij} u_i$).
(See for example,  Bj\"orner and Brenti \cite[\S 4.3]{BjornerBrenti} and
Eriksson \cite{Eriksson}.)
One {\it wins} the game 
when one reaches $u$ having 
all nonnegative coordinates, that is, when $u$ lies in the fundamental chamber $F$.  
The survey by Eriksson \cite{Eriksson} discusses both
the chip-firing and numbers games when played with Cartan matrices for finite root systems 
and points out that generally the two games are 
unrelated \cite[p. 118, \P 2]{Eriksson}.

However, something very special
happens in the sequence \eqref{numbers-game-winning-sequence}.
 Because  $\lambda$ is minuscule, the entries in each $u^{(k)}$
lie in $\{-1,0,1\}$, and one is always firing $u^{(k)}$
at a node $i$ having $u^{(k)}_i=-1$.
Since $C$ has all its diagonal entries $c_{ii}=2$, one can check that
in this setting the involution swapping
$u \leftrightarrow v$ whenever $u+v=\one=\varrho$ sends a 
numbers game configuration $u$ in $\{-1,0,1\}^\ell$ to 
a chip configuration $v$ in $ \{0,1,2\}^\ell$
in such a way that the following diagram commutes:
$$
\begin{array}{rcl}
u & \longleftrightarrow&v\\
\downarrow& &\downarrow \\
u+\alpha_i &\longleftrightarrow &v-\alpha_i\\
\end{array}
$$
Here the left (resp. right) vertical arrow is a valid numbers firing (resp. valid toppling) at node $i$.

One obtains a further interpretation by slightly altering the
sequence \eqref{numbers-game-winning-sequence} in two ways:
\begin{itemize}
\item 
Augment the sequence with $u^{(-1)}:=\lambda$, giving the longer sequence
$$
\begin{array}{cccc}
u^{(-1)}, &u^{(0)}&,u^{(1)},\ldots,u^{(m-1)},&u^{(m)} \\
\Vert &\Vert&  &\Vert \\
\lambda &s_{\alpha^*}(\lambda)&  &\lambda.
\end{array}
$$

\item
Artificially pad each element $u=u^{(k)}=[u_1,\ldots,u_\ell]^t$ in $\ZZ^\ell$
to a vector $\tilde{u} = \tilde{u}^{(k)}:=[u_0,u_1,\ldots,u_\ell]^t$ in $\ZZ^{\ell+1}$
having $u_0:=-\sum_{i=1}^\ell \phi_i u_i$ where $\phi = [1,\phi_1,\dots,\phi_\ell]^t$ gives
the coefficients expressing $\alpha^*(\Phi) = \tilde{\alpha}(\Phi)^\vee=\sum_{i=1}^\ell \phi_i \alpha_i^\vee$ as in \eqref{eq:nullspace}.
This  is {\it forcing} each $\tilde{u}^{(k)}$ 
to lie in $\phi^\perp$.
\end{itemize}
The result 
$\tilde{u}^{(-1)},  \tilde{u}^{(0)}, \cdots, \tilde{u}^{(m)}$
turns out to be a {\it looping} sequence for the
{\it numbers game} played using the {\it extended Cartan matrix} $\tilde{C}(\Phi^\vee)$
for the {\it dual} (!) root system $\Phi^\vee$. 
Looping sequences were characterized by Eriksson \cite[\S3.2]{Eriksson},
and studied further by Gashi and Schedler \cite{GashiSchedler}, 
and by Gashi, Schedler and Speyer \cite{GashiSchedlerSpeyer}.
\end{remark}

\begin{example} \rm \
\label{second-E6-example}
The root system of type  $\ER_6$ is simply laced, so that $\tilde{\alpha}(\Phi)=\tilde{\alpha}(\Phi^\vee)=\alpha^*(\Phi)=\alpha^*(\Phi^\vee)$.
Therefore,  one has $\delta=\delta^\vee=\phi=[1,\phi_1,\ldots,\phi_6]^t$,
with $\phi$ shown here labeling the extended Dynkin diagram for $\tilde{\ER}_6$:
$$
\xymatrix@C=2pt@R=3pt{
& &1 \ar@{-}[d]& & \\
& &2 \ar@{-}[d]& & \\
 1 \ar@{-}[r]  &  2 \ar@{-}[r] &  3 \ar@{-}[r] &  2 \ar@{-}[r] &  1    
}.
$$
The looping sequence $( \tilde{u}^{(k)} )_{k=-1}^{m}$ inside $\phi^\perp$
for the numbers game played with the extended Cartan matrix  $\tilde{C}(\ER_6^\vee)=\tilde{C}(\ER_6)$
for $\ER_6$, is displayed here, to be compared with Example~\ref{first-E6-example}:

$$
\begin{array}{rccccc}

{\tilde u}^{(-1)}=&
\left[
\begin{matrix}
  & &-1& & \\
  & &\ \, 0& & \\
1 &0&\ \,0&0& 0    
\end{matrix}
\right]

&\longrightarrow&

\left[
\begin{matrix}
 & &\ \, 1& & \\
 & &-1& & \\
1&0&\ \, 0&0&0    
\end{matrix}
\right]

&\longrightarrow&

\left[
\begin{matrix}
  & &\ \,0& & \\
  & &\ \,1 & & \\
 1&0&-1&0&  0    
\end{matrix}
\right]

\\
& & &  & & \\
&\quad\uparrow\quad & &  & &\quad\downarrow\quad \\
& & &  & & \\


&
\left[
\begin{matrix}
 & &-1& & \\
 & &\ \, 0 & & \\
 -1   & 1  & \, \ 0  &  0  &  0    
\end{matrix}
\right]

& &  & & 

\left[
\begin{matrix}
 & &\ \,0& & \\
 & &\ \,0 & & \\
 1   &  -1  &\ \,1  &  -1  &  0    
\end{matrix}
\right]

\\
& & &  & & \\
&\quad\uparrow\quad & &  & &\quad\downarrow\quad \\
& & &  & & \\


&\left[
\begin{matrix}
 & &-1& & \\
 & &\ \, 0 & & \\
 0   &  -1  & \ \, 1  &  0  &  0    
\end{matrix}
\right]

 & &  & & 

\left[
\begin{matrix}
 & &\ 0& & \\
 & &\ 0 & & \\
 0   & 1  & \ 0  &  -1  &  0    
\end{matrix}
\right]

\\

& & &  & & \\
&\quad\uparrow\quad & &  & &\quad\downarrow\quad \\
& & &  & & \\


&\left[
\begin{matrix}
 & &-1& & \\
 & &\ \ 1 & & \\
 0   &  0  &  -1  & 1  &  0    
\end{matrix}
\right]

 & &  & & 

\left[
\begin{matrix}
 & &\  \, 0& & \\
 & &\ \, 0 & & \\
 0   &  1  &  -1  &  1  &  -1    
\end{matrix}
\right]

\\
& & &  & & \\
& \quad\uparrow\quad & &  & &\quad\downarrow\quad \\
& & &  & & \\


&\left[
\begin{matrix}
 & &\ 0& & \\
 & &-1 & & \\
 0   &  0  &  \ 0  & 1  &  0    
\end{matrix}
\right]

&\leftarrow&

\left[
\begin{matrix}
 & &\ 0& & \\
 & &-1 & & \\
 0   &  0  & \ 1  &  -1  & 1    
\end{matrix}
\right]

&\leftarrow& 

\left[
\begin{matrix}
 & &\ \, 0& & \\
 & &\ \, 0 & & \\
 0   & 1  &  -1  &  0  & 1    
\end{matrix}
\right]
\end{array}
$$
\end{example}

\begin{example} \rm \
\label{type-C4-example}
The root system $\Phi$ of type $\CR_\ell$ ($\ell \geq 2$) is not simply laced.
Choosing simple roots 
$$
\{\alpha_1=e_1-e_2,\, \alpha_2=e_2-e_3,\, \ldots, \,
     \alpha_{\ell-1}=e_{\ell-1}-e_\ell,\, \alpha_\ell=2e_\ell\}, 
$$
we have that the simple coroots are
$$
\{\alpha^\vee_1=e_1-e_2,\, \alpha^\vee_2=e_2-e_3,\, \ldots, \,
     \alpha_{\ell-1}^\vee=e_{\ell-1}-e_\ell,\, \alpha^\vee_\ell=e_\ell\},
$$
and the fundamental dominant weights for $\CR_\ell$ are
$$
\{\lambda_i=e_1 + \cdots + e_i \mid i=1,\dots, \ell\}.$$  
Moreover, 
$$
\begin{aligned}
\tilde{\alpha}(\Phi)&=2e_1,\\
\tilde{\alpha}(\Phi^\vee)
 &=e_1+e_2=\sum_{i=1}^\ell \delta_i^\vee \alpha^\vee_i
 = \alpha^\vee_1+
  2 \alpha^\vee_2+
  2 \alpha^\vee_3+
  \cdots +
  2 \alpha^\vee_\ell,\\
\alpha^*(\Phi)&=e_1+e_2
 =\sum_{i=1}^{\ell} \phi_i \alpha^\vee_i
 = \alpha^\vee_1+
 2 \alpha^\vee_2+
 2 \alpha^\vee_3+
  \cdots +
 2 \alpha^\vee_\ell
 =\lambda_2.
\end{aligned}
$$
Hence,  $\lambda_1$ is the only minuscule dominant weight, and its associated node is 
darkened in the Dynkin diagram for the root system $\Phi$ of type $\CR_\ell$:
$$
\xymatrix@C=8pt@R=5pt{
\bullet_{\tiny 1}  \ar@{-}[r] & \circ_{\tiny 2} \ar@{-}[r]  & \circ_{\tiny 3} \ar@{-}[r]& 
       \dots \ar@{-}[r] & \circ_{\tiny \ell-1} & \circ_{\tiny \ell.} \ar@{=>}[l]
}
$$
The extended Dynkin diagram for the dual root system 
$\Phi^\vee$ of type $\BR_\ell$ labelled by $\phi$ is given here:
$$
\xymatrix@C=8pt@R=5pt{
1\ar@{-}[r] & 2\ar@{-}[r]  & 2\ar@{-}[r]& \dots \ar@{-}[r] & 2\ar@{=>}[r]&2 \\
            & 1 \ar@{-}[u] &            &                  &             &
}
$$

Taking $\ell=4$, we display on the left the
$C$-topplings showing that 
$\varrho-\lambda_1=\stab_C((\varrho-\lambda_1)+\alpha^*)$,
and on the right, the corresponding looping
sequence $( \tilde{u}^{(k)} )_{k=-1}^{m}$ in $\phi^\perp$ for
the numbers game played with the extended Cartan matrix
$\tilde{C}(\Phi^\vee)=\tilde{C}(\BR_4)$ for the dual root system:
$$
\begin{array}{rcclc}
\varrho-\lambda_1=&
\left[
\begin{matrix}
0 & 1 & 1 & 1\\
\end{matrix}
\right]

& \qquad &

{\tilde u}^{(-1)}=&\left[
\begin{matrix}
1 &\ \ 0 & 0 &0\\
   &-1 &   &
\end{matrix}
\right] \\

 & \downarrow +\alpha^* & & &\downarrow \\

&
\left[
\begin{matrix}
0 & 2 & 1 & 1\\
\end{matrix}
\right]

& \qquad &

&\left[
\begin{matrix}
1 &-1 & 0 &0\\
   & \ \ 1 &   &
\end{matrix}
\right] \\

 & \downarrow & & &\downarrow \\

&
\left[
\begin{matrix}
1 & 0 & 2 & 1\\
\end{matrix}
\right]

& \qquad &

&\left[
\begin{matrix}
0 &\ \, 1 & -1 &0\\
   &\ \, 0 &   &
\end{matrix}
\right] \\

 & \downarrow & & &\downarrow \\

&
\left[
\begin{matrix}
1 & 1 & 0 & 2\\
\end{matrix}
\right]

& \qquad &

&\left[
\begin{matrix}
0 &\ \,  0 & 1 &-1\\
    &\ \,  0 &   &
\end{matrix}
\right] \\

 & \downarrow & & &\downarrow \\

&
\left[
\begin{matrix}
1 & 1 & 2 & 0\\
\end{matrix}
\right]

& \qquad &

&\left[
\begin{matrix}
0 &\ \, 0 & -1 &1\\
   &\ \, 0 &   &
\end{matrix}
\right] \\

 & \downarrow & & &\downarrow \\

&
\left[
\begin{matrix}
1 & 2 & 0 & 1\\
\end{matrix}
\right]

& \qquad &

&\left[
\begin{matrix}
0 & -1 & 1 & 0\\
   & \ \  0 &   &
\end{matrix}
\right] \\

 & \downarrow & & &\downarrow \\

&
\left[
\begin{matrix}
2 & 0 & 1 & 1\\
\end{matrix}
\right]

& \qquad &

&\left[
\begin{matrix}
-1 &\ \ 1 & 0 &0\\
   & -1 &   &
\end{matrix}
\right] \\

 & \downarrow & & &\downarrow \\

&
\left[
\begin{matrix}
0 & 1 & 1 & 1\\
\end{matrix}
\right]

& \qquad &

&\left[
\begin{matrix}
1 &\ \ 0 & 0 &0\\
   & -1 &   &
\end{matrix}
\right]. \\

\end{array}
$$

\end{example}

\begin{section}{McKay quivers}
\label{McKay-quiver-section} \end{section}

We now discuss another source of avalanche-finite matrices from McKay quivers 
and their associated matrices.
We first define the quivers and then review some of their basic properties,
most of which can be found, 
for example, in Steinberg \cite[\S1(2)]{Steinberg}. \medskip

\begin{definitions} \rm \ 
For a complex representation 
$
\gamma:G \rightarrow \GL_n(\CC)
$
of a finite group $G$, 
denote its {\it character} by $\chi_\gamma: G \rightarrow \CC$. 
The {\it McKay quiver} 
of $\gamma$ is the digraph $(Q_0,Q_1)$ whose
node set $Q_0=\{\chi_0,\chi_1,\ldots,\chi_\ell\}$  
is the set of inequivalent irreducible complex $G$-representations $\chi_i$, 
with the convention that $\chi_0=\one_G$ is the trivial
representation, and whose arrow set $Q_1$ has $m_{ij}$ arrows from $\chi_i$ to $\chi_j$ if
one has the irreducible expansions:
\begin{equation}
\label{McKay-equation}
\chi_\gamma\cdot \chi_i  \cong \sum_{j=0}^\ell m_{ij} \chi_j.
\end{equation}
Record the coefficients  as the  matrix 
$M=(m_{ij})$  in $\ZZ^{(\ell+1) \times (\ell+1)}$, 
and define the
\begin{itemize}
\item 
 {\it extended McKay-Cartan matrix} $\tilde{C}:=nI-M$, where $I$ is the identity matrix of size $\ell+1$, 
\item 
 {\it McKay-Cartan matrix} $C$ as the submatrix of $\tilde{C}$
obtained by removing the row and column corresponding to $\chi_0$.
\end{itemize}
\end{definitions}

\begin{proposition}
\label{McKay-review-prop}
Fix a representation 
$
\gamma: G \rightarrow \GL_n(\CC)
$
of a finite group $G$.
\begin{itemize}
\item[{\rm (a)}] A full set of 
orthogonal eigenvectors for $\tilde{C}$ is given by
the column vectors in the character table of $G$,
$$
\delta^{(g)}:=
\left[ \chi_0(g),\chi_1(g),\ldots,\chi_\ell(g) \right]^t,
$$
as $g$ ranges over a set of conjugacy class representatives.
These vectors satisfy the eigenvector equations
\begin{equation}\label{eq:eigeneqn}
\tilde{C} \delta^{(g)} = 
\left(n-\chi_\gamma(g)\right) \cdot \delta^{(g)}.
\end{equation}
\item[{\rm (b)}]
In particular, the nullspace of $\tilde{C}$ contains
\begin{equation}
\label{delta-definition}
\delta^{(e)}=
 \left[ \chi_0(e),\chi_1(e),\ldots,\chi_\ell(e) \right]^t,
\end{equation}
where $e$ is the identity element of $G$, and  a basis for this nullspace is
given by the column vectors
$\{ \delta^{(g)} \}$ indexed by $G$-conjugacy class representatives $g$
lying in the normal subgroup $\ke(\gamma)$ of $G$.  In particular, when $\gamma$ is
faithful (injective), the vector $\delta^{(e)}$ is a basis for the nullspace.
\item[{\rm (c)}]
Consequently, $\tilde{C}$ has rank at most $n-1$, with
equality if and only if $\gamma$ is faithful.
\end{itemize}
\end{proposition}
\begin{proof}
For (a),  the eigenvector equation 
$M \delta^{(g)} = \chi_\gamma(g) \cdot \delta^{(g)}$, which is equivalent to \eqref{eq:eigeneqn},
follows from evaluating both sides of \eqref{McKay-equation} on $g$. 
The rest of (a) is a consequence of the orthogonality of the columns of
the character table.

For (b), use the well-known fact that $\chi_\gamma(g)=n$ if and only if $\gamma(g)=\text{id}_{\CC^n}$: \  if $\gamma(g)$ has eigenvalues
$\zeta_1,\ldots,\zeta_n$ and if $\chi_\gamma(g)=\sum_{i=1}^n \zeta_i=n$,  
then 
$
n=\vert \chi_\gamma(g) \vert 
= \left\vert \sum_{i=1}^n \zeta_i \right\vert
\leq \sum_{i=1}^n  \vert \zeta_i \vert =n,
$
with equality forcing all the $\zeta_i$ to be equal by the Cauchy-Schwartz inequality, 
and hence,  all equal to $1$.

Then (c) is immediate from (b).
\end{proof}

In defining $C$ from $\tilde{C}$,
the choice of the row/column indexed by $\chi_0$, compared to the row/column indexed by any other one-dimensional representation, affects $C$ only
up to similarity, as shown by
the first part of the next proposition.  Recall that  
one-dimensional $G$-representations are simply group homomorphisms 
$\varphi$ in $\widehat{G}=\Hom(G,\CC^\times)$, and for every  irreducible character $\chi_i$ of
$G$,  the complex conjugate $g \mapsto \overline{\chi_i}(g) := \overline{\chi_i(g)} = \chi_i(g^{-1})$ is also an irreducible character.

\begin{proposition}
\label{extended-Cartan-symmetries}
Fix a representation $\gamma: G \rightarrow \GL_n(\CC)$ as before.
\begin{itemize}
\item[{\rm (a)}]
For $\varphi$ in $\widehat{G}$, define a permutation $i \mapsto \varphi(i)$ on $\{0,1,2,\ldots,\ell\}$ so that 
$\chi_i \cdot \varphi = \chi_{\varphi(i)}$.  Then
$$
\tilde{C}_{ij} = \tilde{C}_{\varphi(i)\varphi(j)}.
$$
\item[{\rm (b)}]
Define an involution $i \mapsto i^*$ on $\{0,1,2,\ldots,\ell\}$ so
that $\overline{\chi_i} \cong \chi_{i^*}$.  Then 
$$
\tilde{C}_{ij}=\tilde{C}_{j^*i^*}.
$$
\item[{\rm (c)}]
The involution on $\ZZ^{\ell+1}$ sending $e_i$ to $e_{i^*}$ fixes the vector $\delta^{(e)}$ and induces an isomorphism $\ker(\tilde{C}) \cong \ker(\tilde{C}^t)$.  
In particular, when $\gamma$ is faithful, 
$\delta^{(e)}$ spans both $\ker(\tilde{C})$ and $\ker({\tilde{C}^t})$.
\end{itemize}
\end{proposition}
\begin{proof}
For (a), note that multiplying \eqref{McKay-equation} by $\varphi$ gives
$$
\begin{aligned}
\left(\chi_\gamma \cdot \chi_i\right) \cdot \varphi
& = \sum_{j=0}^\ell m_{ij} \chi_j \cdot \varphi  = \sum_{j=0}^\ell m_{ij} \chi_{\varphi(j)} \\
\chi_\gamma \cdot \left(\chi_i \cdot \varphi\right) 
& = \sum_{j=0}^\ell m_{\varphi(i)j} \chi_{j},
\end{aligned}
$$
implying $m_{ij} = m_{\varphi(i)\varphi(j)}$ for all $i,j$. 

For (b), use
\begin{align}
\label{multiplicities-as-averages}
m_{ij}
&=\frac{1}{|G|} \sum_{g \in G} 
\chi_\gamma(g)\chi_i(g)\chi_j(g^{-1}) 
=\frac{1}{|G|} \sum_{g \in G}  \chi_\gamma(g)
        \overline{ \chi_i}(g^{-1})\overline{\chi_j}(g) \\
& =
        \frac{1}{|G|} \sum_{g \in G} \chi_\gamma(g)
        \chi_{j^*}(g) \chi_{i*}(g^{-1})
=m_{j^*i^*}. \qedhere
\end{align}

For (c), note that $\delta^{(e)}_i= \chi_i(e)=\overline{\chi_i}(e)=\delta^{(e)}_{i^*}$,
and the rest follows from (b).

\end{proof}

Proposition~\ref{extended-Cartan-symmetries}\,(b) has a convenient
rephrasing.  Let $P$ be the $\ell \times \ell$ permutation matrix
for the involution $i \leftrightarrow i^*$ having $\overline{\chi_i} \cong \chi_{i^*}$, 
restricted to the nontrivial irreducible $G$-characters  $\{\chi_1,\ldots,\chi_\ell\}$, so
$P = P^{-1}= P^t$.  Then 
\begin{equation}
\label{transpose-contragredient-symmetry-reformulated}
C^t=PCP=P^t C P.
\end{equation}
This, combined
with an observation of Steinberg \cite{Steinberg},  will prove very useful in showing that $C$ is an avalanche-finite matrix.
To state the ideas in \cite[(3)]{Steinberg}, recall that any square matrix $C$ in $\RR^{\ell \times \ell}$, whether symmetric or not, gives rise to a quadratic form $Q(x)=x^t C x$ on $\RR^\ell$.
The {\it radical} of $Q$ is $\rad(Q):=\{ x \in \RR^\ell: Q(x)=0\}$.
The form $Q$ is said to be  {\it nonnegative semidefinite} if $Q(x) \geq 0$ for all $x$ in $\RR^\ell$,
and {\it positive definite} if additionally $\rad(Q)=\{\zero\}$.

\begin{proposition}\cite[(3)]{Steinberg}
For a representation $\gamma: G \rightarrow \GL_n(\CC)$, consider the quadratic form
$\tilde{Q}(y):=y^t \tilde{C} y$ on $\RR^{\ell+1}$ defined by the extended McKay-Cartan matrix $\tilde{C}$.  
Then $\tilde{Q}$ is nonnegative semidefinite, with $\rad(\tilde{Q}) \supseteq \RR \delta^{(e)}$,
and equality holds  if and only if $\gamma$ is faithful.
\end{proposition}

\begin{corollary}
\label{positive-definite-corollary}
For any representation $\gamma: G \rightarrow \GL_n(\CC)$,  the quadratic form
$Q(x):=x^t C x$ on $\RR^{\ell}$ defined by the McKay-Cartan matrix $C$
is also nonnegative semidefinite, and positive definite whenever $\gamma$ is faithful.
\end{corollary}
\begin{proof}
Include $\RR^\ell \hookrightarrow \RR^{\ell+1}$
by sending $x=[x_1,\ldots,x_\ell]^t \mapsto y=[0,x_1,\ldots,x_\ell]^t$.
Then the quadratic form $Q$ is the restriction of $\tilde{Q}$ 
from $\RR^{\ell+1}$ to $\RR^\ell$.
Therefore $Q$ is nonnegative semidefinite  with
$$
\rad(Q)=\RR^\ell \cap \rad(\tilde{Q})=\RR^\ell \cap \RR \delta^{(e)} =\{\zero\}. \qedhere
$$
\end{proof}

This leads to our proof of 
Theorem~\ref{McKay-Cartan-matrices-are-toppling-theorem}, 
which we restate here.
\vskip.1in
\noindent
{\bf Theorem~\ref{McKay-Cartan-matrices-are-toppling-theorem}.} {} The McKay-Cartan matrix $C$ of a faithful representation $ \gamma: G \hookrightarrow \GL_n(\CC)$ of a finite
group $G$ is an avalanche-finite matrix. 

\vskip.05in
\noindent
\begin{proof}
By Proposition~\ref{toppling-equivalences-prop}(iii),  it suffices to check $C+C^t$ is positive definite.
As $Q(x):=x^tCx$ is positive definite by
Corollary~\ref{positive-definite-corollary}, and $C^t=P^t C P$ by
\eqref{transpose-contragredient-symmetry-reformulated},  then since $P$ is a permutation matrix,  
$$
Q'(x):=x^tC^tx= (P^tx)^tP^tCP(P^tx) =x^tP^tCPx=(Px)^tC(Px)
$$
is positive definite.
Thus $x^t(C+C^t)x=Q(x)+Q'(x)$ is positive definite;  that is,
$C+C^t$ is positive definite.
\end{proof}

\begin{definition} \rm \ 
\label{McKay-critical-group-definition}
Given a faithful representation $\gamma: G \hookrightarrow \GL_n(\CC)$, 
with McKay-Cartan and extended McKay-Cartan matrices $C, \tilde{C}$, 
define its {\it critical group} in any of the following four ways, which are
equivalent by  Proposition~\ref{Cartan-extended-Cartan-prop} 
and Proposition~\ref{extended-Cartan-symmetries}\,(c):
$$
\begin{aligned}
\KG&:=\coker(C^t) \left( =\ZZ^\ell/\im(C^t) =\KC\right), \\ 
&:=(\delta^{(e)})^\perp/\im(\tilde{C}^t), \\
&:=\ZZ^{\ell+1}/(\ZZ e_0+ \im(\tilde{C}^t), \  \text{ or } \\
\ZZ \oplus \KG&:=\coker(\tilde{C}^t).
\end{aligned}
$$
\end{definition}

\begin{example} \rm \
\label{first-A4-example}
Let $G=\AAA_4$ be the alternating subgroup of the symmetric group $\SSS_4$, and consider
its faithful irreducible representation $\gamma: \AAA_4 \rightarrow \SO_3(\RR)\subset \GL_3(\CC)$ as the rotational symmetries
of a regular  tetrahedron, where $\chi_\gamma$ is labeled $\chi_3$ in the 
character table below,  and  $\omega:=\exx^{2\pi i/3}$:

\begin{center}
\begin{tabular}{|c||c|c|c|c|}\hline
  $\AAA_4$     &$e$& $\begin{matrix} (123), (134),\\(142),(243),\end{matrix}$ 
                                & $\begin{matrix} (132),(143)\\(124),(234)\end{matrix}$
                                & $\begin{matrix} (12)(34), (13)(24), \\(14)(23)\end{matrix}$ \\ \hline\hline
$\chi_0$ & $1$ & $1$ & $1$ & $\ \, 1$\\ \hline
$\chi_1$ & $1$ & $\omega$ &$\omega^2$ & $\ \, 1$ \\ \hline
$\chi_2$ & $1$ & $\omega^2$ &$\omega$ & $\ \, 1$\\ \hline
$\chi_3$ & $3$ & $0$ & $0$ & $-1$ \\ \hline
\end{tabular}
\end{center} 
\noindent
The matrices $M, \tilde{C}, C$ associated with $\gamma$ are 
$$
M = 
\bordermatrix{
~     & \chi_0 & \chi_1 & \chi_2 & \chi_3 \cr
\chi_0 & 0 & 0 & 0 & 1 \cr
\chi_1 & 0 & 0 & 0 & 1 \cr
\chi_2 & 0 & 0 & 0 & 1 \cr
\chi_3 & 1 & 1& 1 & 2
}, \quad
\tilde{C} = 
\bordermatrix{
~     & \chi_0 & \chi_1 & \chi_2 & \chi_3 \cr
\chi_0 &\ \,  3 & \ \, 0 & \ \, 0 &-1 \cr
\chi_1 & \ \, 0 & \ \, 3 &\ \,  0 &-1 \cr
\chi_2 & \ \, 0 &\ \,  0 & \ \, 3 &-1 \cr
\chi_3 &-1&-1&-1 &\ \, 1
}, \quad
C = 
\bordermatrix{
~     & \chi_1 & \chi_2 & \chi_3 \cr
\chi_1 & \ \, 3 & \ \, 0 &-1 \cr
\chi_2 & \ \, 0 & \ \, 3 &-1 \cr
\chi_3 &-1&-1 & \ \, 1
}      
$$
Note that $\tilde{C}$ is unchanged by simultaneous cyclic permutations of 
the rows/columns indexed by $\chi_0, \chi_1, \chi_2$, as predicted by
Proposition~\ref{extended-Cartan-symmetries}(a),
since $\{\chi_0,\chi_1,\chi_2\}=\widehat{G} \cong \ZZ/3\ZZ$.

Since $\det(C)=3$, one concludes that $\gamma$ has 
critical group 
$$
\KG:=\coker(C^t: \ZZ^3 \rightarrow \ZZ^3)=\ZZ^3/\im(C^t) \cong \ZZ/3\ZZ.
$$
Toppling with $C$ gives these superstable configurations
$u$ and corresponding (via Theorem~\ref{superstable-recurrent-duality})
recurrent configurations $v^C-u$, where $v^C=[2,2,0]^t$:
\begin{center}
\begin{tabular}{|c||c|c|c|}\hline
\text{superstables}& $[0,0,0]^t$ & $[1,0,0]^t$ & $[0,1,0]^t$ \\ \hline
\text{recurrents} &    $[2,2,0]^t$ & $[1,2,0]^t$ & $[2,1,0]^t$ \\ \hline
\end{tabular} 
\end{center}

\end{example}

\begin{remark} \rm \
The equivalence of the various definitions of $\KG$ given in
Definition~\ref{McKay-critical-group-definition}
is sensitive to the choice of  the row/column 
removed  to  form $C$ from $\tilde{C}$.
For example, in the representation $\gamma: \AAA_4 \rightarrow \SO_3(\RR)$ considered in Example~\ref{first-A4-example},  deleting the row/column corresponding to $\chi_0$ gives  $\coker(\tilde{C}^t) \cong \ZZ \oplus \ZZ/3\ZZ$.   However,  the submatrix of $\tilde{C}$ obtained by
striking out the row/column indexed by $\chi_3$  is the $3 \times 3$ matrix $3I$, whose
cokernel is isomorphic to  $(\ZZ/3\ZZ)^3$.   Hence, toppling with respect to this matrix gives $27$ distinct recurrent configurations, which are all the integer configurations of the form $ [a_1,a_2,a_3]^t$ such that $0 \leq a_i  \leq 2$.
\end{remark}
 
\subsection{A burning configuration 
for McKay-Cartan matrices} \quad \

For a faithful representation $\gamma: G \hookrightarrow \GL_n(\CC)$,
the $0^{th}$ row vector
$
[m_{00},m_{01},m_{02},\ldots,m_{0\ell}]^t 
$
of $M$ gives the (nonnegative) coefficients in the irreducible expansion
$\chi_\gamma= \sum_{j=0}^{\ell} m_{0j} \chi_j$.
Next we observe that the projection  
$\pi: \ZZ^{\ell+1} \rightarrow \ZZ^\ell$ which forgets the $0^{th}$ coordinate
sends it to a vector
$$
b_0:=[m_{01},m_{02},\ldots,m_{0\ell}]^t 
$$
which is a burning configuration for $C$.

\begin{proposition}
In the above setting, $b_0$ is a burning configuration for $C$.
\end{proposition}  

\begin{proof}   It is apparent that $b_0$ lies in $\NN^\ell$, because the multiplicities $m_{ij}$ of the irreducible summands  are nonnegative.  Now we check that $b_0$ satisfies conditions 
(i) and (ii)  in 
Definition~\ref{burning-configuration-definition} of a burning configuration.

For (i),  since $\delta^{(e)} \in \ker(\tilde{C})$ and $\delta^{(e)}_0=1$,
the matrix $\tilde{C}$ has its $0^{th}$ row $\tilde{C}_{0*}$ lying 
in the $\ZZ$-span of the other rows:  
$
\tilde{C}_{0*}=-\sum_{i=1}^\ell \delta^{(e)}_i \tilde{C}_{i*}.
$
Applying $\pi: \ZZ^{\ell+1} \rightarrow \ZZ^\ell$ gives
the second equality here
\begin{equation}
\label{McKay-burning-equation}
b_0
= - \pi(\tilde{C}_{0*})
= \sum_{i=1}^\ell \delta^{(e)}_i \pi(\tilde{C}_{i*})
= \sum_{i=1}^\ell \delta^{(e)}_i C_{i*}
\end{equation}
which shows that $b_0$ lies in $\im(C^t)$.

 For (ii),   Burnside's theorem says that for a faithful representation 
$\gamma$  of a finite group $G$ on $V = \CC^n$, 
 any  irreducible character $\chi_j$ of $G$ occurs in $V^{\otimes m}$ for some $m$  (in fact,
 for  $0 \leq m \leq |\gamma(G)|$ by Brauer's strengthening of that result  (\cite[Thm. 9.34]{Curtis-Reiner}) - see also
 \cite[Problem 4.12.10]{EtingofEtAl}).
This implies there is a directed path with $m$ steps from $\chi_0$ to $\chi_j$ in the McKay quiver  determined by
$\gamma$;   hence, it also gives such a directed path from some node in $\supp(b_0)$ to $\chi_j$ in the digraph $D(C)$
 with at most 
$m-1$ steps.  Consequently, $b_0$ is a burning configuration for $C$.   
\end{proof} 

\begin{remark}
Equation \eqref{McKay-burning-equation}
shows that when stabilizing $v+b_0$ to $v$ for any recurrent configuration $v$,
node $i$ will topple $\delta^{(e)}_i$ times for each $i=1,2,\ldots,\ell$, giving
a total of $\sum_{i=1}^\ell \delta^{(e)}_i$ topplings.
\end{remark}

\begin{example} \rm \
\label{second-A4-example}  For  $G = \AAA_4$ as in Example 
\ref{first-A4-example},  one has $\delta^{(e)} = [1,1,1,3]^t$.   Then 
$$
b_0=[0,0,1]^t = 
C^t [1,1,3]^t
$$   
is a burning configuration, and if $b_0$ 
is added to any of the recurrent configurations $v = [2,2,0]^t,\ [1,2,0]^t$ or $[2,1,0]^t$,  we 
obtain back $v$ after $5 = 1+1+3$ topplings.   
\end{example}

\subsection{Ring and rng structures from tensor product}
\label{rng-section}

As we explain next, besides their additive structure, the 
abelian groups $\KC$ coming from McKay-Cartan
matrices carry a {\it multiplicative} structure 
as a {\it rng} (=ring without unit).
This transpires by viewing $\KC$ as an {\it ideal} inside
the ring $\coker(\tilde{C})$ viewed as a certain quotient of the usual
{\it (virtual) representation ring} $R(G)$ for $G$.

For a finite group $G$ with irreducible complex characters
$\Irr(G)=\{\one_G=\chi_0,\chi_1,\ldots,\chi_\ell\}$,
recall that direct sum $\oplus$ and tensor product $\otimes$ of $G$-representations 
correspond to pointwise addition and pointwise multiplication of characters, 
considered as functions $G \rightarrow \CC$.

\begin{definitions} \rm \
For a finite group $G$,
the {\it ring of virtual characters} or  {\it representation ring} $R(G)$
is a commutative $\ZZ$-algebra whose additive structure
is a free $\ZZ$-module $\ZZ^{\ell+1}$ having the ordered $\ZZ$-basis
 $\{e_0,e_1,\ldots,e_\ell\}$
in bijection with $\Irr(G)$, and
whose multiplication extends $\ZZ$-linearly the rule
$e_i e_j =  \sum_{k=0}^\ell c_k e_k$
if
$
\chi_i \cdot \chi_j = \sum_{k=0}^\ell c_k \chi_k
$ 
as a pointwise equality of functions $G \rightarrow \CC$.
The unit element $1$ of $R(G)$ is $e_0$.
The {\it (virtual) degree} function is a $\ZZ$-algebra homomorphism
defined via
$$
\begin{array}{rcl}
\deg: R(G) &\longrightarrow &\ZZ\\
e_i &\longmapsto &\delta^{(e)}_i (=\chi_i(e)).
\end{array}
$$
One can think of this as taking the dot product  with $\delta^{(e)}$
for vectors in $\ZZ^{\ell+1}$.
\end{definitions} 

For a faithful representation $\gamma:G \hookrightarrow GL_n(\CC)$
of degree $n$, 
with irreducible decomposition 
$\chi_\gamma=\sum_{j=0}^\ell m_{0j} \chi_j$,
its corresponding element $e_\gamma:=\sum_{j=0}^\ell m_{0j} e_j$
in $R(G)$ has $\deg(e_\gamma)=n$. Hence the element $n-e_\gamma$ in $R(G)$
has degree $0$, and generates a principal subideal 
$\lang n-e_\gamma\rang$ inside the ideal $\ker(\deg)=(\delta^{(e)})^\perp$. 

\begin{definition} \rm \
In the above context,
define the quotient ring
$$
R(\gamma):=R(G)/ \lang n-e_\gamma\rang.
$$
Since $\lang n-e_\gamma\rang \subset \ker(\deg)$,   there is an induced
$\ZZ$-algebra homomorphism $\overline{\deg}:R(\gamma) \longrightarrow \ZZ$.
Define its kernel, the ideal within the quotient ring $R(\gamma)$,  
$$
I(\gamma):=\ker(\overline{\deg}: R(\gamma) \longrightarrow \ZZ).
$$
\end{definition}

\begin{proposition}
\label{ring-and-rng-prop}
For a faithful representation $\gamma: G \hookrightarrow GL_n(\CC)$,
there are additive isomorphisms
\begin{align}
\label{Pic-as-ring}
R(\gamma)& \cong  \coker(\tilde{C}) 
         \qquad \left( \cong \ZZ \oplus \KC \right)\\
\label{Pic-zero-as-rng}
I(\gamma) & \cong \coker (C)  
         \qquad \left( \cong \KC \right),
\end{align}
which endow
\begin{itemize}
\item $\ZZ \oplus \KC$ with the extra structure of
a $\ZZ$-algebra as $R(\gamma)$, and 
\item $\KC$ with the extra structure of a ring-without-unit (rng),
as the ideal $I(\gamma)$ in $R(\gamma)$.
\end{itemize}
\end{proposition}

\begin{proof}
The isomorphism \eqref{Pic-as-ring}
follows since the matrix $\tilde{C}^t$ expresses multiplication by 
$n-e_\gamma$ in the ring $R(G)$, when using the ordered $\ZZ$-basis 
$\{e_0,e_1\ldots,e_\ell\}$. 
Then \eqref{Pic-zero-as-rng} comes from restricting this 
multiplication by $n-\chi_\gamma$ to 
its action on the ideal $(\delta^{(e)})^\perp=\ker(\deg)$ within $R(G)$,
and using the equivalent definition $\KC=(\delta^{(e)})^\perp/\im(\tilde{C})$
from Definition~\ref{McKay-critical-group-definition}.
\end{proof}

Here are three examples of these ring and rng structures.

\begin{example} \rm \
\label{third-A4-example}
Continue Example~\ref{first-A4-example}
of the alternating group $G=\AAA_4$.
It has irreducible characters $\Irr(G)=\{\chi_0,\chi_1,\chi_2,\chi_3\}$,
and its character table shown there allows one to
check these relations:
$$
\left\{ 
\begin{matrix}
\chi_{k}&=\chi_{1}^k,\\
\chi_k \cdot \chi_3&=\chi_3,
\end{matrix}
\right. 
\quad \text{ for }k=0,1,2,
\qquad \text{ and } \qquad 
\begin{aligned}
\chi_1^3&=\one_G=\chi_0\\
\chi_3^2&=2\chi_3+\chi_0+\chi_1+\chi_2.
\end{aligned}
$$
Therefore letting $x:=e_1, y:=e_3$, with $\deg(x)=1, \deg(y)=3$,  one has
$$
R(G) \cong \ZZ[x,y]/ \lang x^3-1, \,\, xy-y, \,\, y^2-(2y+1+x+x^2)\rang.
$$

Now consider the faithful representation 
$\gamma: G=\AAA_4 \hookrightarrow \SO_3(\RR) \subset \GL_3(\CC)$
from Example~\ref{first-A4-example} representing $G=\AAA_4$ 
as the rotational symmetries of a
regular tetrahedron.  Since $\chi_\gamma=\chi_3$,  one has
$$
\begin{aligned}
R(\gamma) 
&:=R(G)/\lang 3-y\rang\\
& \cong \ZZ[x,y]/\lang x^3-1, \,\, xy-y, \,\, y^2-(2y+1+x+x^2), \,\, 3-y\rang\\
& \cong \ZZ[x]/\lang x^3-1, \,\, 3(x-1), \,\, x^2+x-2\rang\\
& \cong \ZZ[x]/\lang 3(x-1), \,\, (x-1)^2\rang\\
& \cong \ZZ[u]/ \lang 3u, \,\, u^2\rang,
\end{aligned}
$$
where the last isomorphism comes from a change of variable $u:=x-1$.
In fact, $u$ principally generates the ideal $I(\gamma)$ in $R(\gamma)$,
and additively one has
$$
R(\gamma) \cong \ZZ \oplus I(\gamma)
$$
in which
$$
I(\gamma) = ( \ZZ/3\ZZ ) u \cong \ZZ/3\ZZ,
$$
with trivial rng (multiplicative) structure on $I(\gamma)$ since $u^2=0$.
\end{example}

\begin{example} \rm \
\label{cyclic-McKay-example}
The cyclic group $G=\langle g:  g^m =e\rangle \cong \ZZ/m\ZZ $ 
has  $m$ irreducible characters 
$\Irr(G)=\{\chi_0,\chi_1,\chi_2, \cdots,  \chi_{m-1}\}$,
all one-dimensional of the form 
$\chi_k(g^j)=\omega_m^{jk}$ with $\omega_m:=e^{2\pi i/ m}$.
Hence $\chi_{k}=\chi_{1}^k$ for $k=0,1,2,\ldots,m-1$
with $\chi_1^{m}=\chi_0=\one_G$, and 
$$
R(G) \cong \ZZ[x]/\lang x^m-1\rang,
$$
where  $x:=e_1$  and   $\deg(x)=1$.
Choosing the faithful representation 
$\gamma: G\hookrightarrow \SL_2(\CC)$
that sends 
$
g \longmapsto 
\left[
\begin{smallmatrix} 
  \omega_m & 0 \\  0 & \omega^{-1}_m
\end{smallmatrix}
\right],
$
one has $\chi_\gamma=\chi_1+\chi_{m-1}=\chi_1+\chi_1^{m-1}$,
and hence $e_\gamma=x+x^{m-1}$.
Therefore,
$$
\begin{aligned}
R(\gamma) 
 &:=R(G)/\lang 2-(x+x^{m-1})\rang\\
 &\cong \ZZ[x]/\lang x^m-1,\,\, x^{m-1}+x-2\rang\\
 &\cong \ZZ[x]/\lang x^2-2x+1,\,\ x^{m-1}+x-2\rang\\
 & \cong \ZZ[x]/ \lang x^2-2x+1, \,\, m(x-1)\rang\\
 & \cong \ZZ[u]/ \lang u^2, \,\, mu\rang,
\end{aligned}
$$
where 
\begin{itemize}
\item the second isomorphism used the fact that in $\ZZ[x]$
dividing $x^m-1$ by $x^{m-1}+x-2$  
leaves a remainder of $-(x^2-2x+1) =  -(x-1)^2$, when $m \geq 3$,
\item the third isomorphism used the fact that in $\ZZ[x]$
dividing $x^{m-1}+x-2$ by  $x^2-2x+1$ leaves a remainder of $m(x-1)$,
\end{itemize}
and the final isomorphism comes from the change of variable $u:=x-1$.
Thus,  additively one has
$$
\begin{aligned}
R(\gamma) &\cong \ZZ \oplus I(\gamma) \\
I(\gamma) &\cong \left( \ZZ/m \ZZ \right) u \cong \ZZ/m\ZZ
\end{aligned}
$$
with trivial rng (multiplicative) structure on $I(\gamma)$ since $u^2=0$.
\end{example} 

\begin{example} \rm \
\label{cyclic-non-SL-example}
To illustrate some nontrivial multiplicative structure on $I(\gamma)$,
we consider a different family of faithful representations for the same cyclic
group $G$ of order $m$.  Fix $n \geq 1$, and let
$\gamma: G \hookrightarrow \GL_n(\CC)$ send $g$ to  the $n \times n$ scalar matrix $\omega_m  I$.
Therefore,  $\chi_\gamma=n \chi_1$, and $e_\gamma=nx$, so that
$$
\begin{aligned}
R(\gamma) 
&:=R(G)/\lang n-nx\rang \\
& \cong \ZZ[x]/ \lang  x^m -1, \,\, n-nx\rang\\
& \cong \ZZ[u]/ \lang (u+1)^m-1, \,\,nu\rang, 
\end{aligned}
$$
where the last isomorphism comes from our usual change of variable $u:=x-1$.
Again, $u$ principally generates the ideal $I(\gamma)$ in $R(\gamma)$,
and additively,  
$$
R(\gamma) \cong \ZZ \oplus I(\gamma)
$$
in which
$$
\begin{aligned}
I(\gamma) &= ( \ZZ/n\ZZ ) u \oplus
( \ZZ/n\ZZ ) u^2 \oplus \cdots \oplus
( \ZZ/n\ZZ ) u^{m-1} 
&\cong (\ZZ/n\ZZ)^{m-1}
\end{aligned}
$$
with rng (multiplicative) structure determined
by  $(u+1)^m-1=0$, that is,
$
u^m =-\sum_{k=1}^{m-1} \binom{m}{k} u^k.
$
\end{example} 

\begin{section}{Relation to the abelianization}
\label{abelianization-section}\end{section}

Our goal here is to prove Theorem~\ref{abelianization-theorem},
relating the critical group $\KG$ to the {\it abelianization}
$
G^{\ab}:=G/[G,G],
$
whenever $\gamma$ is a faithful representation that maps  $G$ into the 
{\it special linear group} $\SL_n(\CC)$ 
of complex $n\times n$-matrices of determinant 1.

Recall that every group homomorphism in $\widehat{G}$
factors as a composite 
$G \twoheadrightarrow G^\ab \rightarrow \CC^\times$,
that is, $\widehat{G} \cong \widehat{G^{\ab}}$.
Hence for $G$ finite, $\widehat{G}$ is Pontrjagin dual to
the abelianization $G^{\ab}$, 
and (non-canonically) isomorphic to it.  \medskip

\begin{definition}
For $\gamma: G \rightarrow \GL_n(\CC)$,
let $\det_\gamma:=\det \circ \gamma$ in $\widehat{G}$,
that is, $\det_\gamma(g):=\det(\gamma(g))$.
\end{definition}

The following is a more precise version of Theorem~\ref{abelianization-theorem}.
In stating it, we are considering $\ZZ^{\ell+1}$ and $\widehat{G}$    
as abelian groups under addition and pointwise multiplication, respectively.

\begin{theorem}
\label{McKay-surjection-theorem}

For a faithful representation $\gamma: G \    \hookrightarrow \  \SL_n(\CC)$ 
of a finite group $G$, the homomorphism 
$$
\begin{array}{rcl}
\ZZ^{\ell+1} & \overset{\pi}{\longrightarrow}& \widehat{G} \\
e_i  &\longmapsto & \det_{\chi_i}
\end{array}
$$
induces a surjective homomorphism of abelian groups 
$
\KG:=\ZZ^{\ell+1}/ \big( \ZZ e_0+ \im(\tilde{C}^t)\big )
 \twoheadrightarrow  \widehat{G}.
$
\end{theorem}

\begin{proof}  
Surjectivity of $\pi$ follows because every homomorphism
$\varphi$ in $\widehat{G}$
is itself a one-dimensional irreducible character $\chi_i$  of $G$  for some $i$, and 
hence $\pi(e_i)=\varphi$.  
Since $\KG =\ZZ^{\ell+1}/ \big( \ZZ e_0+ \im(\tilde{C}^t)\big)$,
it suffices to show that
$e_0$ and $\im(\tilde{C}^t)$ lie in $\ke(\pi)$.
For $e_0$ this is clear, since $\chi_0=\one_G$ and
$\det_{\one_G}=\one_G$.

To show that $\im(\tilde{C}^t) \subset \ke(\pi)$, 
we compare the expressions for $\det_\gamma$ on the two sides of

\begin{equation}
\label{repeated-McKay-relation}
\chi_\gamma \cdot \chi_i = \sum_{j=0}^\ell m_{ij} \chi_j.
\end{equation}    
On the left side of \eqref{repeated-McKay-relation}, note that for any
 two linear operators $T_i: U_i \rightarrow U_i$  acting on finite-dimensional vector spaces $U_i$  for $i=1,2$, 
 we  have
\begin{equation}
\label{det-of-tensors}
\det{}_{U_1 \otimes U_2}(T_1 \otimes T_2) =
 \det{}_{U_1}(T_1)^{\dim(U_2)} \det{}_{U_2}(T_2)^{\dim(U_1)}.
\end{equation}
\noindent
Consequently, for any two genuine characters $\chi, \psi$ of $G$, one has
\begin{equation}
\label{det-of-tensors-consequence}
\det{}_{\chi \cdot \psi} 
   = \left(  \det{}_\chi \right)^{\psi(e)}
      \left(  \det{}_\psi \right)^{\chi(e)}.
\end{equation}

Since $\gamma: G \rightarrow \SL_n(\CC)$ 
means that $\det_\gamma(-)=\one_G$, \eqref{det-of-tensors}, this shows that 
$$
\det{}_{\chi_\gamma \cdot \chi_i} = \left(\det{}_{\chi_i}\right)^n 
                      \left(\det{}_\gamma\right)^{\delta^{(e)}_i} 
                    = \left(\det{}_{\chi_i}\right)^n.
$$
Comparing this with the right side of \eqref{repeated-McKay-relation},  we conclude that for each $i=0,1,2,\ldots,\ell$, one has
$$
\left(\det{}_{\chi_i}\right)^n = \prod_{j=0}^\ell \left(\det{}_{\chi_j} \right)^{m_{ij}}.
$$
This says that row $i$ of $\tilde{C}$ (= column $i$ of $\tilde{C}^t$) lies in $\ke(\pi)$ for each $i$, 
and  hence $\im(\tilde{C}^t) \subset \ke(\pi)$.
\end{proof}

\begin{example} \rm \
\label{fourth-A4-example}
Example~\ref{first-A4-example} 
considered a faithful representation  
$\gamma: G=\AAA_4 \hookrightarrow \SL_3(\CC)$
with
$$
\KG \cong \ZZ/3\ZZ \cong G^{\ab}.
$$
\end{example}

\begin{example} \rm \
\label{second-cyclic-McKay-example}
Example~\ref{cyclic-McKay-example},
considered a faithful representation 
$\gamma: G=\ZZ/m\ZZ \hookrightarrow \SL_2(\CC)$
with
$$
\KG \cong  \ZZ/m\ZZ \cong G(=G^{\ab}).
$$
\end{example}

\begin{example} \rm \
On the other hand, Example~\ref{cyclic-non-SL-example}
considered a different 
family of faithful representations 
$\gamma: G=\ZZ/m\ZZ \hookrightarrow \GL_n(\CC)$
that sent $g$  to  $\omega_m I \in \GL_n(\CC)$, with
$$
\KG \cong (\ZZ/n\ZZ)^m.
$$
Note that $\gamma(G) \subset \SL_n(\CC)$ 
if and only if $m$ divides $n$, which is
exactly the same condition under which $\KG\cong (\ZZ/n\ZZ)^m$ can
surject onto $\ZZ/m\ZZ =G (=G^\ab)$,
as Theorem~\ref{McKay-surjection-theorem} would predict.
\end{example}

Theorem~\ref{McKay-surjection-theorem} interacts in an interesting way with multiplication in $R(G)$.

\begin{proposition}
\label{products-annihilated}
The surjection $\pi$ from Theorem~\ref{McKay-surjection-theorem}, considered
as a homomorphism 
$$
(R(G),+) \longrightarrow \widehat{G}
$$
annihilates all products $xy$ with $x,y,$ in $I(G)$, that is, it factors through
the quotient $( R(G) /I(G)^2, + )$.
\end{proposition}
\begin{proof}
Note that $I(G):=\ker(\deg)$ is the $\ZZ$-linear span of the elements
$\{e_i - \delta^{(e)}_i\}_{i=1,2,\ldots,\ell}$:
$$
x=\sum_{i=0}^\ell c_i e_i \text{ lies in } \ker(\deg)
\quad \Longleftrightarrow \quad
\sum_{i=0}^\ell c_i \delta^{(e)}_i=0
\quad \Longleftrightarrow \quad 
x=\sum_{i=0}^\ell c_i (e_i-\delta^{(e)}_i).
$$
Therefore it suffices to show that $\pi$ annihilates all products of the form
$(e_i - \delta^{(e)}_i )(e_j - \delta^{(e)}_j)$:
$$
\begin{aligned}
\pi\left( (e_i - \delta^{(e)}_i )(e_j - \delta^{(e)}_j) \right)
&= \pi\left( e_i e_j  
                - ( \delta^{(e)}_i e_j +  \delta^{(e)}_j e_i) 
                 + \delta^{(e)} \delta^{(e)}_j) \right) \\
&=\left( \det_{\chi_i \cdot \chi_j} \cdot \one_G^{\delta^{(e)}_i \delta^{(e)}_j} \right)
      / \left( (\det_{\chi_i})^{\delta^{(e)}_j}  (\det_{\chi_j})^{\delta^{(e)}_i} \right) 
 =\one_G,
\end{aligned}
$$
where the last equality used  identity \eqref{det-of-tensors-consequence}.
\end{proof}

\begin{corollary}
\label{isomorphism-implies-trivial-rng}
The surjection $\pi$ from Theorem~\ref{McKay-surjection-theorem}
induces a surjection of abelian groups
$$
\KG \cong (I(\gamma),+) \twoheadrightarrow \widehat {G}
$$
which annihilates all products in $I(\gamma)^2$.  In particular, whenever
there is an isomorphism $\KG \cong {\widehat G}$, all products in $I(\gamma)$ vanish,
that is, $I(\gamma)^2=0$.
\end{corollary}
\begin{proof}
Since $\pi(e_i-\delta^{(e)}_i)=\det_{\chi_i}=\pi(e_i)$, the map
$\pi$ is surjective when restricted from $R(G)$ to $I(G)$.  Since
$\pi$ descends to the quotent $I(\gamma)$ of $I(G)$, the rest follows 
from Proposition~\ref{products-annihilated}.
\end{proof}

\begin{question} \rm \
\label{isomorphism-question}
For which faithful representations $\gamma: G \hookrightarrow \SL_n(\CC)$ of a finite group $G$
is the surjection in Theorem~\ref{McKay-surjection-theorem} an {\it isomorphism}
$\KG \cong \widehat{G}$?
\end{question}

\noindent
The answer for {\it abelian groups} will be given in Proposition~\ref{abelian-isomorphism-characterization} below. 

We mention here two general
reductions in Question~\ref{isomorphism-question}.

\begin{proposition}
\label{reductions-proposition}
The critical group $\KG$ of a 
representation $\gamma:G  \rightarrow \GL_n(\CC)$
and the ring and rng structures on $R(\gamma), I(\gamma)$
are unchanged by

\begin{itemize}
\item[{\rm (a)}]
adding copies of the trivial representation, that is,
replacing $\chi_\gamma \mapsto \chi_\gamma+ d \cdot \chi_0$, or
\item[{\rm (b)}]
precomposing with any group
automorphism $\sigma: G \rightarrow G$,
that is, replacing $\gamma \mapsto \gamma \circ \sigma$.
\end{itemize}
\end{proposition}
\begin{proof}
For (a), note that replacing $\chi_\gamma$ with 
$\chi_{\gamma'}:=\chi_\gamma+ d \cdot \chi_0$,
replaces $M$ with $M'=M+dI$, and hence $\tilde{C}$ with 
$
\tilde{C}'=(n+d)I-(M+dI)=nI-M=\tilde{C}.
$
It also replaces $n-e_\gamma$ with $(n+d)-(e_\gamma+d)=n-e_\gamma$,
so that $R(\gamma), I(\gamma), \coker(\tilde{C}),$ and $\coker(C)$ are
unchanged.
For (b), note that precomposing $G$-representations with $ \sigma$
permutes $\{\chi_0,\chi_1,\ldots,\chi_\ell\}$,
replacing $\tilde{C}$ by $P\tilde{C} P^{-1}$
for some permutation matrix $P$, and inducing 
simultaneous ring automorphisms of $R(G)$ and $R(\gamma)$.
\end{proof}


\subsection{Example:  McKay correspondence for  
subgroups of $\SL_2(\CC)$}

For a finite subgroup $G$  of
$\SL_2(\CC)$,  the action of $G$ on $\CC^2$ (by matrix multiplication)  gives a faithful
representation $\gamma$.
McKay originally observed in \cite{McKay} 
(see also Steinberg \cite[\S1(4)]{Steinberg})
that the McKay quiver $(Q_0,Q_1)$ has 
\begin{compactenum}
\item[$\bullet$] {\it undirected} arcs, 
implying $m_{ij}=m_{ji}$,
\item[$\bullet$] no multiple edges,  meaning $m_{ij}=m_{ji} \in \{0,1\}$, 
\item[$\bullet$] no self-loops,  meaning $m_{ii}=0$,
\item[$\bullet$] $\tilde{C}=2 I -M$,  
the extended Cartan matrix of a simply-laced root system,
as in Section~\ref{extended-Cartan-matrix-section}.
\end{compactenum}

\noindent
In particular, 
$\KG=\ZZ^\ell/\im(C^t) = P(\Phi)/Q(\Phi)$ is the 
{\it fundamental group} of the finite root system $\Phi$,
the weight lattice modulo the root lattice,
discussed in Section~\ref{root-order-section}.

Note that here $\delta^{(e)}=\delta=\delta^\vee$,
where $\delta^{(e)}$ was defined in \eqref{delta-definition},
and $\delta, \delta^\vee$ were defined in 
Section~\ref{root-order-section}:
$\Phi$ being simply laced implies
$\delta^\vee=\delta$ is the basis for the left or right-nullspace of 
$\tilde{C}$, and since $\gamma$ is faithful, Proposition~\ref{McKay-review-prop}(c)
tells us that this nullspace is one-dimensional,
spanned by $\delta^{(e)}$ from \eqref{delta-definition}.  
Since $\delta_0=1=\delta_0^{(e)}$, it must be that  $\delta=\delta^{(e)}$.

We have the following more precise version of 
Theorem~\ref{McKay-abelianization-theorem}.

\begin{theorem}
\label{classical-McKay-abelianization-theorem} For $G$ a finite subgroup of $\SL_2(\CC)$
and  $\gamma: G \hookrightarrow \SL_2(\CC)$ (the natural representation), 
the surjection in Theorem~\ref{McKay-surjection-theorem}
is an isomorphism 
$
\KG \cong \widehat{G}.
$
Thus the critical group $\KG$ is Pontrjagin dual 
(so non-canonically isomorphic) to
the abelianization $G^{\ab}$.
\end{theorem}

\begin{proof}[Proof of Theorem~\ref{classical-McKay-abelianization-theorem}.]
It suffices to show $|\KG|=|\widehat{G}|$.
As discussed above,  $\KG=P(\Phi)/Q(\Phi)$,
and hence $|\KG|=f$,
the {\it index of connection} for $\Phi$.
Section~\ref{root-order-section}
asserted
that $f-1$ is the number of indices
$i=1,2,\ldots,\ell$ for which $\delta^\vee_i=1$ in 
\eqref{minuscule-dominant-characterization}.
Since $\delta^\vee=\delta=\delta^{(e)}$,
it must be that $f-1$ is the number of $i$
for which $\delta^{(e)}_i=1$, that is, the 
number of one-dimensional characters in $\widehat{G} \setminus \{\one_G\}$.
Thus $|\widehat{G}|=f=|\KG|$.
\end{proof}


\hspace{-.2cm} \begin{tabular}{|c|c|c|c|} \hline
$\begin{matrix} \hbox{\rm root} \\
\hbox{\rm system} \ \Phi 
\end{matrix}$& $\begin{matrix} \hbox{\rm presentation of} \\ G \subset \SL_2(\CC)\end{matrix}$ & 
$\begin{matrix}\KG  \cong P(\Phi)/Q(\Phi) \\
\hspace{-.38cm} \cong G^{\ab} \end{matrix}
$ 
  & $\begin{matrix}  \hbox{\rm affine diagram labeled by} \\ \delta=\delta^\vee=\delta^{(e)}
  \end{matrix}$ \\ \hline\hline
$\AR_{m-1} $& $\langle{r: r^m}=1 \rangle$ & $\ZZ/ m\ZZ$ & 
\xymatrix@C=10pt@R=3pt{
& &1\ar@{-}[dll]\ar@{-}[drr]& & \\
1 \ar@{-}[r] & 1 \ar@{-}[r] & \cdots  \ar@{-}[r] &  1 \ar@{-}[r] &  1  
}
\\ \hline
$\DR_{m+2}$ & $\langle {r,s,t: r^2=s^2=t^m=rst \rangle}$  &$\ZZ/4\ZZ \  \text{ if }m  \text{ odd}$  & \\
       &                                                                    &$\left(\ZZ/2\ZZ\right)^2 \text{ if } m \text{ even}$ & 
\xymatrix@C=10pt@R=3pt{
1 \ar@{-}[dr]& & & & &1 \ar@{-}[dl] \\
& 2 \ar@{-}[r] &  2 \ar@{-}[r] & \cdots \ar@{-}[r]  &  2  &  \\ 
1 \ar@{-}[ur]& & & & &1 \ar@{-}[ul] 
} 
\\ \hline
$\ER_6$  &$\langle {r,s,t: r^2=s^3=t^3=rst} \rangle$ & $\ZZ/3\ZZ$  &
\xymatrix@C=10pt@R=3pt{
& &1 \ar@{-}[d]& & \\
& &2 \ar@{-}[d]& & \\
 1 \ar@{-}[r]  &  2 \ar@{-}[r] &  3 \ar@{-}[r] &  2 \ar@{-}[r] &  1    
}  
\\ \hline
$\ER_7$  &$\langle {r,s,t: r^2=s^3=t^4=rst}\rangle$ & $\ZZ/2\ZZ$ & 
\xymatrix@C=10pt@R=3pt{
& & &2 \ar@{-}[d]& & & \\
 1 \ar@{-}[r]  &  2 \ar@{-}[r] &  3\ar@{-}[r] & 4 \ar@{-}[r] &  3 \ar@{-}[r] &  2 \ar@{-}[r] &  1    
} 
\\ \hline
$\ER_8$  &$\langle {r,s,t: r^2=s^3=t^5=rst} \rangle$ & $0$  & 
\xymatrix@C=10pt@R=3pt{
& & & & & 3 \ar@{-}[d]& & \\
 1 \ar@{-}[r]  &  2 \ar@{-}[r] &  3\ar@{-}[r] & 4 \ar@{-}[r] &  5 \ar@{-}[r] &  6 \ar@{-}[r] &  4  \ar@{-}[r] & 2  
} 
\\\hline
\end{tabular}

\vskip.1in
\noindent
For example, $\AR_{m-1}$ in the table corresponds
to the representation of $\ZZ/m\ZZ$ from 
Example~\ref{cyclic-McKay-example}.

\begin{remark} \rm \ 
\label{Steinberg-abelianization-remark}
We explain next  how Theorem~\ref{classical-McKay-abelianization-theorem}
is Pontrjagin dual 
to a result of Steinberg \cite[\S1(6)]{Steinberg}.

A finite, crystallographic, irreducible root system $\Phi$ has associated
to it various compact real Lie groups $\GGG$ that all share
the same Lie algebra $\ggg$ and root system $\Phi$.  The two extremes
among them are the {\it adjoint group} $\GGG_{\ad}$ and
its simply-connected universal cover $\GGG_{\sc} \rightarrow \GGG_{\ad}$,
a Galois covering with covering group
$$
\pi_1(\GGG_{\ad}) \cong  P(\Phi^\vee)/Q(\Phi^\vee) \cong \mathfrak Z(\GGG_{\sc})
$$
where $\mathfrak Z(\GGG_\sc)$ denotes the center of the group $\GGG_\sc$.
This explains the 
terminology ``fundamental group''; see Br\"ocker-tom Dieck
\cite[Chap. V \S 7, Thms. 7.1, 7.16]{BrockerTomDieck} and
Bourbaki \cite[Chap. IX \S 9, Thm.2ff]{Bourbaki789}.
For a finite group $G \subset \SL_2(\CC)$, Steinberg \cite[\S1(6)]{Steinberg}
describes 
an isomorphism $G^\ab \overset{\sim}{\longrightarrow}  \mathfrak Z(\GGG_{\sc})$.
Applying the contravariant Pontrjagin duality functor $A \mapsto \widehat{A}$
then gives this disguised version of the isomorphism in 
Theorem~\ref{classical-McKay-abelianization-theorem}:
$$
\begin{array}{ccc}
\widehat{ \mathfrak Z(\GGG_\sc)} 
&\overset{\sim}{\longrightarrow} &\widehat{G^\ab}\\
\Vert & &\Vert \\
P(\Phi)/Q(\Phi)& &\widehat{G}\\
\Vert& & \\
\KG
\end{array}
$$
\end{remark}

\begin{remark} \rm \ 
Theorem~\ref{classical-McKay-abelianization-theorem}
disagrees slightly 
with Brylinski \cite[Cor.~5.9]{Brylinski},
which says that the {\it index of connection} $f:=|\KG|$ is
the {\it exponent} of the abelianization $G^\ab$.
For example, in the McKay correspondence for type $\DR_\ell$ with $\ell$ even,  
the above table shows $\KG=G^\ab=(\ZZ/2\ZZ)^2$,
so that one has $f=4$, but the exponent of $G^\ab$ is $2$.  
\end{remark}

\begin{remark} \rm \
In McKay's original setting of $G \hookrightarrow \SL_2(\CC)$, 
the interpretation of
minuscule dominant weights from Proposition~\ref{minuscule-dominant-properties}\,(a)
enables us to reinterpret the  assertion of 
Theorem~\ref{Cartan-recurrents-theorem}(i) as follows:
the (nonzero) superstable configurations for $C$ are exactly
the basis vectors $e_i$ corresponding to the (nontrivial)
$1$-dimensional characters $\chi_i$ of $G$.  
In light of Theorem~\ref{abelianization-theorem},  this suggests the following question.

\begin{question} \rm \
Given a faithful representation $G \hookrightarrow \SL_n(\CC)$, and
its McKay-Cartan matrix $C$,
do the basis vectors $e_i$ corresponding to the nontrivial
one-dimensional characters $\chi_i$ of $G$ always form a subset
of the superstable configurations for $C$?
\end{question}

\end{remark}

\subsection{Example:  abelian groups}
\label{abelian-groups-section}

We explain here why any faithful representation 
$\gamma: G \   \hookrightarrow \ \GL_n(\CC)$ of a finite {\it abelian} group $G$
always  has critical group $\KG$ equal to the usual critical group $K(D)$ for the 
directed graph $D$ which is its McKay quiver, as in Remark~\ref{digraph-remark}.

For $G$ abelian, all irreducible representations of $G$
are one-dimensional, that is, 
$\{\chi_0,\chi_1,\ldots,\chi_\ell\}=\widehat{G}$.
Evaluating the characters on both sides of 
$\chi_\gamma \cdot \chi_i = \sum_{j=0}^\ell m_{ij} \chi_j$  at $e$ 
shows that  
$
n=\sum_{j=0}^\ell m_{ij}
$
is the outdegree of $\chi_i$ for each $i$  in the McKay quiver considered
as a digraph $D$.
Thus,  the extended Cartan matrix $\tilde{C}$ is the same as the 
Laplacian matrix $\tilde{L}$ for $D$.  
In fact, $D$ can be considered as the {\it Cayley digraph}
for $\widehat{G}$ in which each potential generator $\chi_j$ is given
multiplicity $m_{0j}$ if $\chi_\gamma = \sum_{j=0}^\ell m_{0j} \chi_j$.  One
can check that the set $\{\chi_j: m_{0j} \geq 1\}$ actually {\it does}
generate $\widehat{G}$ if and only if $\gamma$ is faithful.
In particular, all vertices in $D$ look the same, up to translation
by elements of $\widehat{G}$, so that
the choice of the row/column in $\tilde{C}$ to strike out 
in forming $C$ is immaterial.  

\begin{example} \rm \
Examples~\ref{cyclic-McKay-example} and \ref{cyclic-non-SL-example}
both had $G=\langle g: g^m =e\rangle \cong \ZZ/m\ZZ$,
and therefore $\tilde{C}=\tilde{L}$
for a circulant digraph $D$ on vertex set $\{0,1,\ldots,m-1\}$.
\begin{itemize}
\item In the classical McKay case of type $\AR_{m-1}$ in  Example~\ref{cyclic-McKay-example},
where
$
g  \mapsto  
  \left[ \begin{smallmatrix} 
   \omega_m & 0 \\ 0 & \omega_m^{-1} 
  \end{smallmatrix} \right],
$
one has $\chi_\gamma = \chi_1 + \chi_{m-1}$,
and hence the digraph $D$ has both arrows  $\chi_i \rightarrow \chi_{i+1}$
and $\chi_{i+1} \rightarrow \chi_{i}$ for each $i=0,1,2,\ldots,m-1$ (subscripts 
modulo $m$).
\item In Example~\ref{cyclic-non-SL-example},
where
$
g  \mapsto \omega_m I,
$
one has $\chi_\gamma = n \chi_1$,
and hence the digraph $D$ has $n$ copies of the same arrow 
$\chi_i \rightarrow \chi_{i+1}$ for each $i=0,1,2,\ldots,m-1$  (subscripts modulo $m$).
\end{itemize}
\end{example}

Interestingly, in the case where $G$ is abelian, one can resolve
Question~\ref{isomorphism-question}(i).
The authors are grateful to S. Koplewitz for providing a proof
of the following proposition.

\begin{proposition}
\label{abelian-isomorphism-characterization}
Let $G$ be a finite abelian group,
and $\gamma: G \hookrightarrow \SL_n(\CC)$
be a faithful representation with no copies of $\one_G$.
Then $\KG \cong \widehat{G}$ if and only if $n=2$, that is,
if and only if $G\cong \ZZ/ m\ZZ$ with $G \subset \SL_2(\CC)$
as in McKay's type $\AR_{m-1}$.
\end{proposition}
\begin{proof}[Sketch proof.]
Recall that Theorem~\ref{abelianization-theorem} 
provides a surjection $\KG \twoheadrightarrow \widehat{G}$.
Since for digraphs $D$, one has an interpretation 
(see Remark~\ref{digraph-remark}) for $|K(D)|$ as the number of arborescences
in $D$ directed toward the vertex $\chi_0$,  this proposition can be rephrased as follows:   
Given a 
\begin{itemize}
\item finite abelian group $(A,+)$ \  (corresponding to $A=\widehat{G}$),
\item a multiset $a_1,\ldots,a_r$ of $A\setminus \{0\}$ \ 
(corresponding to $\chi_\gamma=\sum_{i=1}^r \chi_{a_i}$
not containing the character $\chi_0= \one_G$),
\item with  $\{a_i\}_{i=1}^r$ \  generating \  $A$ \ 
(corresponding to $\chi_\gamma$ being faithful),
\item and $\sum_{i=1}^r a_i=0$ 
(corresponding to $\gamma(G) \subset \SL_n(\CC)$), 
\end{itemize}
then the associated Cayley digraph for $(A,(a_i)_{i=1}^r)$ has
the number of arborescences directed toward $0$ strictly greater than $|A|$
whenever $r \geq 3$.  

To see this, use the fact (see, e.g., \cite[Thm. 3.3]{GallianWitte}) 
that any Cayley digraph $D$ for an abelian group $A$ has a (directed)
Hamilton path 
$
(b_{|A|} \rightarrow \cdots \rightarrow b_2 \rightarrow b_1 \rightarrow b_0:=0),
$
and any element $a_j$ of 
the generating multiset $\{a_i\}_{i=1}^r$ can be chosen as the label
$a_j=b_1-b_0$ on the last step $b_1 \rightarrow b_0$.
The idea is then to construct a map $f$ from $A$ 
to the set of arborescences that
takes $b_i$ to an arborescence containing the subpath 
$(b_i \rightarrow \cdots \rightarrow b_2 \rightarrow b_1 \rightarrow b_0:=0)$
but not containing the arc $b_{i+1} \rightarrow b_i$.  Roughly speaking,
 there is flexibility to do this,  since the condition
$\sum_{i=1}^r a_i=0$ implies that the generator $a_j:=b_{i+1}-b_i$
is redundant for generating $A$, and hence partial arborescences
can be completed to full arborescences avoiding arcs 
(such as $b_{i+1} \rightarrow b_i$) labeled by $a_j$.  
This gives an injective map $f$, which one 
can show is {\it not} surjective for $r \geq 3$.
\end{proof}


\begin{thebibliography}{XX}
\bibitem{AsadiBackman}
A. Asadi and S. Backman,
Chip-firing and Riemann-Roch theory for directed graphs,
{\it Electronic Notes in Discrete Math.} {\bf 38} (2011), 63--68. 

\bibitem{BakerShokrieh}
M.~Baker and F.~Shokrieh,{\it Chip-firing games, potential theory on graphs, and
 spanning trees},
J. Combin. Theory Ser. A {\bf 120} (2013), no. 1, 164--182.  

\bibitem{Biggs} 
  N. Biggs, {\it Chip-firing and the critical group of a graph}, J. Algebraic Combin.  {\bf 9} (1999)(1), 25--45.
  
\bibitem{BjornerBrenti}
A. Bj\"orner, and F. Brenti, 
Combinatorics of Coxeter groups. 
{\it Graduate Texts in Mathematics} {\bf 231},  Springer, New York, 2005. 

\bibitem{BjornerLovasz}  A. Bj\"{o}rner and L. Lov\'{a}sz. {\it Chip-firing games on directed graphs.} J. Algebraic Combin. {\bf 1}  (1992),  no. 4, 305--328. 

\bibitem{Bourbaki456}
N. Bourbaki,
Lie groups and Lie algebras. Chapters 4--6.
{\it Elements of Mathematics}, Springer-Verlag, Berlin, 2002.

\bibitem{Bourbaki789}
N. Bourbaki, 
Lie groups and Lie algebras. Chapters 7--9.
{\it Elements of Mathematics}, Springer-Verlag, Berlin, 2005.


\bibitem{BrockerTomDieck}
T. Br\"ocker and T. tom Dieck, 
Representations of compact Lie groups.
{\it Graduate Texts in Mathematics} {\bf 98}. 
Springer-Verlag, New York,  1995.

\bibitem{Brylinski}
J.-L. Brylinski,
A correspondence dual to McKay's,
{\tt  arXiv:9612003v2}.


\bibitem{BTW} P. Bak, C. Tang, and K. Wiesenfeld, Self-organized criticality, Phys. 
Rev. A, {\bf 38}, (1988), no. 1, 364--374. 

\bibitem{ChapmanGarciaEtAl}
S. Chapman, R. Garcia, L.D.~Garcia-Puente, M.E. Malandro, and K. W. Smith, 
Algebraic and combinatorial aspects of sandpile monoids on directed graphs,
 {\it J. Combin. Theory Ser. A} {\bf 120} (2013),  245--265.
 
\bibitem{Curtis-Reiner}
 C.W. Curtis and I. Reiner,   Methods of representation theory. Vol. I. With applications to finite groups and orders,  Pure and Applied Mathematics. A Wiley-Interscience Publication. John Wiley \& Sons, Inc., New York, 1981.
 
 \bibitem{Dhar} 
D. Dhar, Self-organised critical state of the sandpile automaton models, 
Phys. Review Letters {\bf 64} (1990), no. 14, 1613--1616. 

\bibitem{DummitFoote}
D.S. Dummit and R.M. Foote, 
Abstract algebra, 3rd edition.
John Wiley \& Sons, Inc., Hoboken, NJ, 2004.

\bibitem{Eriksson}
K. Eriksson, 
Node firing games on graphs, 
Jerusalem Combinatorics '93, 117--127,
{\it Contemp. Math.} {\bf 178}, Amer. Math. Soc., Providence, RI, 1994. 

\bibitem{EtingofEtAl}
P. Etingof, O. Golberg, S. Hensel, T. Liu, A. Schwendner,  D. Vaintrob, and E. Yudovina, 
Introduction to representation theory.
{\it Student Mathematical Library} {\bf 59}. 
American Mathematical Society, Providence, RI, 2011.


\bibitem{Gabrielov}
A. Gabrielov,
Asymmetric abelian avalanches and sandpile, preprint 93-65, 
Math. Sciences Institute, 
Cornell University, 1993.

\bibitem{GallianWitte}
J. Gallian and D. Witte,
A survey: Hamiltonian cycles in Cayley graphs.
{\it Discrete Math.} {\bf 51} (1984), 293--304. 

\bibitem{GashiSchedler}
Q.R.~Gashi and T. Schedler, 
On dominance and minuscule Weyl group elements. 
{\it J. Algebraic Combin.} {\bf 33} (2011), 383--399. 

\bibitem{GashiSchedlerSpeyer}
Q.R.~Gashi, T. Schedler, and D. Speyer, 
Looping of the numbers game and the alcoved hypercube.
{\it J. Combin. Theory Ser. A} {\bf 119} (2012), 713--730. 

\bibitem{Green}
R.M. Green,
Combinatorics of minuscule representations.
{\it Cambridge Tracts in Mathematics} {\bf 199}.  
Cambridge University Press, Cambridge, 2013.

\bibitem{GuzmanKlivans}
J. Guzm\'{a}n and C. Klivans.
Chip-firing and energy minimization on M-matrices,
J.~Combin.~Theory Ser.~A {\bf 132} (2015), 14--31. 

\bibitem{HLMPPW}
A. Holroyd, L. Levine, K. Meszaros, Y. Peres, J. Propp, and D. B. Wilson.
Chip-firing and rotor-routing on directed graphs. In and out of equilibrium. 2, 
331--364,
{\it Progr. Probab.} {\bf  60},  Birkh\"auser, Basel, 2008.
 
\bibitem{Humphreys}
J.E.~Humphreys,   
Introduction to Lie algebras
and representation theory, Second printing, revised.
{\it Graduate Texts in Mathematics} {\bf 9},  Springer-Verlag, New York-Berlin 1978. 

 \bibitem{Humphreys2} 
J.E.~Humphreys,   
Reflection groups and Coxeter groups,
{\it Cambridge Studies in Advanced Mathematics} {\bf 29},  Cambridge University Press, Cambridge, 1990.

\bibitem{Kac}
V.G. Kac,
Infinite-dimensional Lie algebras,  3rd edition. 
Cambridge University Press, Cambridge,  1990. 


\bibitem{McKay}
J. McKay, 
Cartan matrices, finite groups of quaternions, and Kleinian singularities.
{\it Proc. Amer. Math. Soc.} {\bf 81} (1981), no.~1, 153--154. 

\bibitem{Merino} C. Merino, The chip firing game and matroid complexes,  Discrete models: combinatorics, computation, and geometry (Paris, 2001), 245�255 (electronic), {\it Discrete Math. Theor. Comput. Sci.} Proc., AA, Maison Inform. Math. Discr�t. (MIMD), Paris, 2001. 

\bibitem{Perkinson}  D. Perkinson, J. Perlman,  and J. Wilmes,  Primer for the 
algebraic geometry of sandpiles, Tropical and non-Archimedean geometry, 211--256,  {\it Contemp.
Math.}  {\bf 605}  Amer. Math. Soc., Providence, RI, 2013. 

\bibitem{Plemmons}
R. J. Plemmons, 
M-matrix characterizations. I. Nonsingular M-matrices, 
{\it Linear Algebra and Appl.} {\bf 18} (1977), no.~2, 175--188. 

 \bibitem{PostnikovShapiro}
A.~Postnikov and B.~Shapiro, Trees, parking functions, syzygies, and deformations of monomial ideals,
Trans. Amer. Math. {\bf 356} (2004), no.~8,   3109--3142. 

\bibitem{Speer}
E.R. Speer, 
Asymmetric abelian sandpile models. 
{\it J. Statist. Phys.} {\bf 71} (1993),  no.~1-2, 61--74. 

\bibitem{Springer}
T.A. Springer, 
Invariant theory.
{\it Lecture Notes in Mathematics} {\bf 585}. 
Springer-Verlag, Berlin-New York, 1977.

\bibitem{Steinberg}
R. Steinberg, 
Finite subgroups of $SU_2$, Dynkin diagrams and affine Coxeter elements.
{\it Pacific J. Math.} {\bf 118} (1985),  no.~2,  587--598. 

 \bibitem{Stembridge}
J.R.~Stembridge, 
The partial order of dominant weights, {\it Adv. Math.} {\bf 136},
(1998), 340--364.


\bibitem{Wagner}
D.G. Wagner,
The critical group of a directed graph,
{\tt arXiv:0010241}.
  
\end{thebibliography}
\end{document}